\documentclass[10pt,psamsfonts]{amsart}
\usepackage{amsmath}
\usepackage{amsthm}
\usepackage{amssymb}
\usepackage{amscd}
\usepackage{amsfonts}
\usepackage{amsbsy}
\usepackage{graphicx}
\usepackage[dvips]{psfrag}
\usepackage{array}
\usepackage{color}
\usepackage{epsfig}
\usepackage{url}
\usepackage{overpic}
\usepackage{epstopdf}
\usepackage{float}
\usepackage{enumerate}
\usepackage{multirow}
\usepackage{rotating}
\usepackage{hyperref}

\newcommand{\R}{\ensuremath{\mathbb{R}}}
\newcommand{\N}{\ensuremath{\mathbb{N}}}
\newcommand{\Q}{\ensuremath{\mathbb{Q}}}

\newcommand{\F}{\ensuremath{\mathcal{F}}}
\newcommand{\G}{\ensuremath{\mathcal{G}}}
\newcommand{\CF}{\ensuremath{\mathcal{F}}}
\newcommand{\z}{\ensuremath{\mathcal{Z}}}

\newcommand{\al}{\alpha}
\newcommand{\be}{\beta}

\newcommand{\s}{\Sigma}

\newcommand{\sgn}{\mathrm{sign}}

\newcommand{\D}{\Delta}
\DeclareMathOperator{\Span}{Span}
\DeclareMathOperator{\Dis}{Dis}
\usepackage{mathtools}
\usepackage[pagewise]{lineno}

\def\e{\varepsilon}

\newtheorem {theorem} {Theorem}

\newtheorem {proposition} [theorem]{Proposition}

\newtheorem {lemma}  [theorem]{Lemma}

\newtheorem {mtheorem} {Theorem}

\usepackage{mathpazo}

\begin{document}
\renewcommand{\arraystretch}{1.5}

\title[Melnikov analysis for planar piecewise linear vector fields with algebraic switching curve  $y^n-x^m=0$]{Melnikov analysis for planar piecewise linear vector fields with algebraic switching curve  $y^n-x^m=0$}

\author[C. C. Santos]{Cintia C. Santos$^1$}
\address{$^1$ Departamento de Matem\'{a}tica, Universidade
	Federal de Vi\c{c}osa,\ Campus Universit\'ario,
	36570-000, Vi\c cosa-MG, Brazil} 
\email{cintiacsantos@ufv.br}

\author[O.A.R. Cespedes]{Oscar A. R. Cespedes$^2$}
\address{$^2$ Departamento de Matem\'{a}tica, Universidade
	Federal de Vi\c{c}osa,\ Campus Universit\'ario,
	36570-000, Vi\c cosa-MG, Brazil} 
\email{oscar.ramirez@ufv.br}



\maketitle

\begin{abstract}
This paper is devoted to the study of  the maximum number of limit cycles, $H(m,n)$, of a planar piecewise linear differential systems with two zones separated by the curve  $y^n-x^m=0$, with $n,m$ being positive integers.  More precisely, we provide a lower estimate of $ H(m,n)$. for all $m,n\in \N$, for piecewise linear perturbations of the linear center using 
 some recent results about Chebyshev systems with positive accuracy and Melnikov Theory. 
\end{abstract}

\section{Introduction}

Non-smooth differential equations have been used extensively to describe the dynamic behavior of a large number of problems in
 in mechanics, electronics, economy and biology, see for example \cite{ColJefLazOlm2017, LliNovRod2017, Sim2010, Teixeira2011}.\\

In the last decades, one of the main topics of the qualitative theory of non-smooth differential systems is determining the number and distribution of limit cycles of a  planar piecewise linear differential system.  There is a large number of papers that address this problem assuming that the commutation curve is a straight line (see, for instance, \cite{ArtLliMedTei2014, BPT13, BraMel2013, FrePonTor2012, FrePonTor2014, GianPli2001, HanZha2010, HuaYan2013, HuaYan2014,   Li2014, LliNovTei2015b, LNT15b, LliPon2012, LliTeiTor2013,LliZha2019}).   Recently, the nonlinear case has been studied by considering that the commutation curve is of the form $y-x^m=0$, with $m$ being a positive integer,  see for example \cite{AndCesCruNov2021, ColJefLazOlm2017, LliNovRod2017, Sim2010, Teixeira2011}.\\

The main objective of this paper is to study the maximum number of limit cycles, $H(m,n)$, of a planar piecewise linear differential systems with two zones separated by the curve  $y^n-x^m=0$, with $n,m$ being positive integers.  More precisely, we provide a lower estimate of $ H(m,n)$ for piecewise linear perturbations of the linear center. For that, we consider the following planar piecewise linear vector field	
\begin{equation}\label{general-system1}
Z(x,y)=\left\{  \begin{array}{lc}
X(x,y) = \left(\begin{array}{c}
y+\displaystyle\sum_{i=1}^{k}\e^iP_i^+(x,y) \\
-x+\displaystyle\sum_{i=1}^{k}\e^iQ_i^+(x,y)
\end{array}\right), & y^n-x^m>0, \\ & \\
Y(x,y) = \left(\begin{array}{c}
y+\displaystyle\sum_{i=1}^{k}\e^iP_i^-(x,y)\\
-x+\displaystyle\sum_{i=1}^{k}\e^iQ_i^-(x,y)
\end{array}\right), & y^n-x^m<0, 
\end{array}
\right.
\end{equation}
where $m,n$ are positive integers, and $P_i^{\pm}$ and $Q_i^{\pm}$ are affine functions given by
\[ \begin{array}{rcl}
P_i^+(x,y) & = &a_{0i}+a_{1i}x+a_{2i}y, \\
P_i^-(x,y) & = &\al_{0i}+\al_{1i}x+\al_{2i}y, \\
Q_i^+(x,y) & = &b_{0i}+b_{1i}x+b_{2i}y, \\
Q_i^-(x,y) & = &\be_{0i}+\be_{1i}x+\be_{2i}y,
\end{array}\]
with $a_{ji},\al_{ji},b_{ji},\be_{ji}\in\R,$ for $i\in\{1,\ldots,k\}$ and $j\in\{0,1,2\}$. Notice that the commutation manifold of  \eqref{general-system1} is $\Sigma=\{(x,y)\in\R^2:\, y^n-x^m=0\}$.\\

In order to provide the lower estimate of $H(m,n)$ we use the Melnikov Theory of a high order,  recently developed for nonsmooth differential systems with nonlinear manifold in \cite{AndCesCruNov2021}. In short, this theory provides a sequence of functions $\Delta_i$, $ i \in\{ 1, 2,..., k\}$, which control, in some sense, the existence of isolated periodic orbits bifurcating from the unperturbed period annulus. These functions are called Melnikov Functions or Bifurcation Functions. \\

 In the sequel,  we denote by $m_\ell(m,n)$  the maximum number of limit cycles of system \eqref{general-system1} which can be obtained using the Melnikov function of order $\ell$.   In Table \ref{tab1}, we summarize  the known values for  $m_{\ell}(m, n)$.  In particular, these previous studies provided $H(2,1)\geq 4,$ $H(3,1)\geq 8,$ $H(m,1)\geq7,$ for $m\geq 4$ even, and  $H(m,1)\geq9,$ for $m\geq 5$ odd. It is important to mention that at this moment in the literature, the only case studied is  when $n=1$ and $m$ is a positive integer.

  \begin{table}[h!]

\centering 

\begin{tabular}{cccccccc}
 &  & \multicolumn{6}{c}{Order $k$} \\ \cline{3-8} 
 & \multicolumn{1}{c||}{} & \multicolumn{1}{c|}{{\bf 1}} & \multicolumn{1}{c|}{{\bf 2}} & \multicolumn{1}{c|}{{\bf 3}} &\multicolumn{1}{c|}{{\bf 4}} &\multicolumn{1}{c|}{{\bf 5}} & \multicolumn{1}{c|}{${\bf 6}$} \\ \cline{2-8} 
\multicolumn{1}{c|}{\multirow{4}{*}{\begin{sideways}Degree $m$\end{sideways}}} & \multicolumn{1}{c||}{{\bf 1}} & \multicolumn{1}{c|}{1} & \multicolumn{1}{c|}{1} & \multicolumn{1}{c|}{2} & \multicolumn{1}{c|}{3}& \multicolumn{1}{c|}{3}& \multicolumn{1}{c|}{3} \\ \cline{2-8} 

\multicolumn{1}{c|}{} & \multicolumn{1}{c||}{{\bf 2}} & \multicolumn{1}{c|}{3} & \multicolumn{1}{c|}{4} & \multicolumn{1}{c|}{4} & \multicolumn{1}{c|}{4}& \multicolumn{1}{c|}{4}& \multicolumn{1}{c|}{4} \\ \cline{2-8} 

\multicolumn{1}{c|}{} & \multicolumn{1}{c||}{ ${\bf 3}$} & \multicolumn{1}{c|}{3} & \multicolumn{1}{c|}{7} & \multicolumn{1}{c|}{7} & \multicolumn{1}{c|}{7}& \multicolumn{1}{c|}{7}& \multicolumn{1}{c|}{$8\leq m_6\leq 10$} \\ \cline{2-8} 

\multicolumn{1}{c|}{} & \multicolumn{1}{c||}{${\bf m\geq 4}$ {\bf even}} & \multicolumn{1}{c|}{4} & \multicolumn{1}{c|}{7} & \multicolumn{1}{c|}{7} & \multicolumn{1}{c|}{7}& \multicolumn{1}{c|}{7}& \multicolumn{1}{c|}{7} \\ \cline{2-8} 

\multicolumn{1}{c|}{} & \multicolumn{1}{c||}{ ${\bf m\geq 5}$ {\bf odd}} & \multicolumn{1}{c|}{3} & \multicolumn{1}{c|}{7} & \multicolumn{1}{c|}{7} & \multicolumn{1}{c|}{7}& \multicolumn{1}{c|}{7}& \multicolumn{1}{c|}{$9\leq m_6\leq 14$} \\ \cline{2-8} 
\end{tabular}
\bigskip
\caption{Known values in the research literature  for  $m_{\ell}(m, 1)$ }\label{tab1}
\end{table}  

 Our main results provide the values $m_{\ell}(m,n)$ for all $m,n\in\mathbb{N}$ and $\ell\in\{1,2,3\}$. The contribution of Theorems \ref{Theorem-Melnikov2}, \ref{Theorem-Melnikov21} and \ref{Theorem-Melnikov22} are summarized in Tables \ref{tab2}, \ref{tab3} and \ref{tab4}. 

\begin{mtheorem}\label{Theorem-Melnikov2} Let $m$ and $n$ be  even positive integers.
The following statements hold for the piecewise linear differential system \eqref{general-system1},
\begin{enumerate}
\item $m_1(m,n)=0$ and $m_2(m,n)=1$, for $\frac{m}{n}=1$;
\item $m_1(m,n)=2$ and $m_2(m,n)=4$, for $\frac{m}{n} \in \left[\frac{1}{5},\frac{1}{3}\right] \cup \left[\frac{1}{2},1 \right) \cup (1,2] \cup [3,5]$;
\item $m_1(m,n)=2$ and $m_2(m,n)=5$, for $\frac{m}{n} \in \left(0,\frac{1}{5} \right) \cup \left(\frac{1}{3},\frac{1}{2} \right) \cup (2,3) \cup (5, \infty)$.
\end{enumerate}
 Consequently,
    \[
    H(m,n) \geq \left\lbrace 
    \begin{array}{ll}
1, \text{ for }\frac{m}{n}=1, \\ 4, \text{ for } \frac{m}{n} \in \left[\frac{1}{5},\frac{1}{3}\right] \cup \left[\frac{1}{2},1 \right) \cup (1,2] \cup [3,5], \\
5, \text{ for } \frac{m}{n} \in \left(0,\frac{1}{5} \right) \cup \left(\frac{1}{3},\frac{1}{2} \right) \cup (2,3) \cup (5, \infty).\\
    \end{array} \right.
    \]

   \end{mtheorem}
 
    \begin{mtheorem}\label{Theorem-Melnikov21}  Let $m$ and $n$ be odd positive integers. 
    	The following statements hold for the piecewise linear differential system \eqref{general-system1},
    	\begin{enumerate}
    		\item $m_1(m,n)=m_2(m,n)=1$ and $m_3(m,n)=2$, for $\frac{m}{n}=1$;
    		\item $m_1(m,n)=2$ and $m_\lambda(m,n)=7$, for $\frac{m}{n} \in \left(\frac{2}{3},\frac{3}{4} \right] \cup \left[\frac{4}{3},\frac{3}{2}\right)$;
    		\item $m_1(m,n)=3$ and $m_\lambda(m,n)=7$, for $\frac{m}{n} \in \left(0,\frac{1}{2}\right) \cup (2, \infty)$;
    		\item $m_1(m,n)=2$ and $m_\lambda(m,n)=8$, for $\frac{m}{n} \in \left(\frac{1}{2},\frac{2}{3} \right) \cup \left(\frac{3}{4}, 1\right) \cup \left(1,\frac{4}{3}\right) \cup \left(\frac{3}{2},2\right)$,
    	\end{enumerate}
where $\lambda\in\{2,3\}$.  Consequently,
    	  \[
    	  H(m,n) \geq \left\lbrace 
    	  \begin{array}{ll}
    	  2, \text{ for }\frac{m}{n}=1, \\
    	  7, \text{ for } \frac{m}{n} \in \left(0,\frac{1}{2}\right) \cup \left(\frac{2}{3},\frac{3}{4}\right] \cup \left[\frac{4}{3},\frac{3}{2}\right) \cup (2, \infty), \\
    	  8, \text{ for } \frac{m}{n} \in \left(\frac{1}{2},\frac{2}{3} \right) \cup \left(\frac{3}{4}, 1 \right) \cup \left(1,\frac{4}{3} \right) \cup \left(\frac{3}{2},2 \right).
    	  \end{array} \right.
    	  \]
    \end{mtheorem}
    
       \begin{mtheorem}\label{Theorem-Melnikov22}  Let $m$ and $n$ be positive integers. If $m$ is even  and $n$ is odd then the following statements hold for the piecewise linear differential system \eqref{general-system1},
\begin{enumerate}
       		\item $m_1(m,n)=3$ and $m_\lambda(m,n)=4$, for $\frac{m}{n}=2$;
       		\item $m_1(m,n)=3$ and $m_\lambda(m,n)=5$, for $\frac{m}{n}=\frac{2}{3}$;
       		\item $m_1(m,n)=3$ and $m_\lambda(m,n)=6$, for $\frac{m}{n} \in  \left(\frac{2}{3},\frac{3}{4} \right) \cup \left[\frac{4}{5},1 \right) \cup \left(1,\frac{4}{3} \right] $;
       		\item $m_1(m,n)=4$ and $m_\lambda(m,n)=6$, for $\frac{m}{n} \in \left(\frac{1}{5}, k_2 \right) \cup \left(\frac{3}{2}, k_3 \right)$;
       		\item $m_1(m,n)=3$ and $m_\lambda(m,n)=7$, for $\frac{m}{n} \in  \left(\frac{1}{2},\frac{2}{3} \right) \cup \left(\frac{3}{4},\frac{4}{5} \right) \cup \left(\frac{4}{3},\frac{3}{2} \right) \cup (k_5, \infty)$;
       		\item $m_1(m,n)=4$ and $m_\lambda(m,n)=7$, for $\frac{m}{n} \in \left(k_1,\frac{1}{5} \right) \cup \left(\frac{1}{3},\frac{1}{2} \right) $;
       		\item $m_1(m,n)=3$ and $7 \leq m_\lambda (m,n) \leq 8$, for $\frac{m}{n} \in  (2,3)$;
       		\item $m_1(m,n)=4$ and $7 \leq m_\lambda(m,n) \leq 8$, for $\frac{m}{n} \in  \left(k_2,\frac{1}{3} \right) \cup (k_3,2) $;
       		\item $m_1(m,n)=3$ and $7 \leq m_\lambda (m,n) \leq 9$, for $\frac{m}{n} \in (k_4,k_5)$;
       		\item $m_1(m,n)=3$ and $6 \leq m_\lambda (m,n) \leq 11$, for $\frac{m}{n} \in (3,k_4)$;
       		\item $m_1(m,n)=4$ and $6 \leq m_\lambda (m,n) \leq 11$, for $\frac{m}{n} \in (0,k_0) \cup (k_0,k_1)$, 
       \end{enumerate}
where $\lambda\in\{2,3\}$ and
  \begin{enumerate}[(i)]
	 	\item  $k_1$ is the irrational root in $\left(0, \frac{1}{5} \right)$ of  $$q_1(x)=2 x^6-153 x^5+894 x^4-1625 x^3+1234 x^2-444 x+48;$$
	 	\item $k_2,k_3,k_5$ are the irrational roots in $\left(\frac{1}{5},\frac{1}{3} \right), \left(\frac{3}{2},2 \right),$ and  $(3,4)$, respectively, of  
  		\[
  		\begin{array}{ll}
  		q_2(x)=&177586560 x^{17}-1447424208 x^{16}+5969663136 x^{15}-29387373904 x^{14}\\&+129626832188 x^{13}-293774330511 x^{12}+102470736381 x^{11}+1027573184492 x^{10}\\&-2645532232771 x^9+3259827826136 x^8-2344796073539 x^7+997977148820 x^6\\&-227233713561 x^5+15757275163 x^4+3311923726 x^3-563207524 x^2\\&+11249208 x+19744;
  		\end{array}
  		\]
	 	
	 	\item  $k_4$  is the irrational root in $(3,4)$ of $$q_3(x)=862 x^5-4831 x^4+8308 x^3-5186 x^2+974 x+1.$$
 \end{enumerate}

 Consequently,
       	    \[
       	    H(m,n) \geq \left\lbrace 
       	    \begin{array}{ll}
       	    4, &\text{ for }\frac{m}{n}=2, \\ 5, & \text{ for }\frac{m}{n}=\frac{2}{3}, \\ 6, &\text{ for } \frac{m}{n} \in (0,k_0) \cup (k_0,k_1)  \cup\left(\frac{1}{5}, k_2 \right) \cup \left(\frac{2}{3},\frac{3}{4} \right) \cup \left[\frac{4}{5},1 \right) \cup \left(1, \frac{4}{3} \right] \cup \left(\frac{3}{2}, k_3 \right) \cup (3,k_4), \\
       	    7, & \text{ for } \frac{m}{n} \in \left(k_1,\frac{1}{5} \right) \cup \left(k_2,\frac{1}{3} \right) \cup \left(\frac{1}{3},\frac{1}{2} \right) \cup \left(\frac{1}{2},\frac{2}{3} \right) \cup \left(\frac{3}{4},\frac{4}{5} \right) \cup \left(\frac{4}{3},\frac{3}{2} \right) \cup (k_3,2) \\ &\hspace{1.5cm} \cup (2,3) \cup (k_4,k_5) \cup (k_5, \infty).\\
       	    \end{array} \right.
       	    \]  
       	\end{mtheorem}

  \begin{table}[H]
  	\begin{center}
  		\begin{tabular}{ccccc}
  			&  & \multicolumn{3}{c}{Order $k$} \\ \cline{3-5} 
  			\multicolumn{1}{c}{}	& \multicolumn{1}{c||}{} & \multicolumn{1}{c|}{{\bf 1}} & \multicolumn{1}{c|}{{\bf 2}} & \multicolumn{1}{c|}{{\bf 3}} \\ \cline{2-5} 
  			\multicolumn{1}{c|}{\multirow{5}{*}{\begin{sideways}${\bf m/n \in}$\end{sideways}}} & \multicolumn{1}{c||}{${\bf \left(0,\frac{1}{2} \right) \cup (2, \infty)}$} & \multicolumn{1}{c|}{3} & \multicolumn{1}{c|}{7} & \multicolumn{1}{c|}{7}  \\ \cline{2-5} 
  					
  					\multicolumn{1}{c|}{} & \multicolumn{1}{c||}{${\bf  \left(\frac{1}{2},\frac{2}{3} \right) \cup \left(\frac{3}{4},1 \right) \cup \left(1,\frac{4}{3} \right) \cup \left(\frac{3}{2},2 \right)}$} & \multicolumn{1}{c|}{2} & \multicolumn{1}{c|}{8} & \multicolumn{1}{c|}{8} \\ \cline{2-5} 
  					
  					\multicolumn{1}{c|}{} & \multicolumn{1}{c||}{ ${\bf  \left(\frac{2}{3},\frac{3}{4} \right] \cup \left[\frac{4}{3},\frac{3}{2} \right)}$} & \multicolumn{1}{c|}{2} & \multicolumn{1}{c|}{7} & \multicolumn{1}{c|}{7} \\ \cline{2-5} 
  					
  					\multicolumn{1}{c|}{} & \multicolumn{1}{c||}{ ${\bf \{1\}}$} & \multicolumn{1}{c|}{1} & \multicolumn{1}{c|}{1} & \multicolumn{1}{c|}{2} \\ \cline{2-5} 
  				\end{tabular}
  				\bigskip
  				\caption{Values for $m_\ell(m,n)$, with $\ell\in\{1,2,3\}$, when $m,n$ are odd positive integers} \label{tab2}
  			\end{center}
  		\end{table}

 \begin{table}[H]
\begin{tabular}{cccc}
	&  & \multicolumn{2}{c}{Order $k$} \\ \cline{3-4} 
\multicolumn{1}{c}{}	& \multicolumn{1}{c||}{} & \multicolumn{1}{c|}{{\bf 1}} & \multicolumn{1}{c|}{{\bf 2}}  \\ \cline{2-4} 
	\multicolumn{1}{c|}{\multirow{4}{*}{\begin{sideways}${\bf m/n}\in$ \end{sideways}}} & \multicolumn{1}{c||}{${\bf  \left(0,\frac{1}{5} \right) \cup \left(\frac{1}{3},\frac{1}{2} \right) \cup (2,3) \cup (5, \infty)}$} & \multicolumn{1}{c|}{2} & \multicolumn{1}{c|}{5}   \\ \cline{2-4} 
	
	\multicolumn{1}{c|}{} & \multicolumn{1}{c||}{${\bf  \left[\frac{1}{5},\frac{1}{3} \right] \cup \left[\frac{1}{2},1 \right) \cup (1,2] \cup [3,5] }$} & \multicolumn{1}{c|}{2} & \multicolumn{1}{c|}{4}  \\ \cline{2-4} 
	
	\multicolumn{1}{c|}{} & \multicolumn{1}{c||}{ ${\bf \{1\}}$} & \multicolumn{1}{c|}{0} & \multicolumn{1}{c|}{1}  \\ \cline{2-4} 
\end{tabular}
\bigskip
\caption{Values for $m_\ell(m,n)$, with $\ell\in\{1,2,3\}$, when $m,n$ are even positive integers}\label{tab3}
\end{table}

	\begin{table}[H]
		\begin{center} 
			\begin{tabular}{ccccc} 
				&  & \multicolumn{3}{c}{Order $k$} \\ \cline{3-5}
				\multicolumn{1}{c}{}	& \multicolumn{1}{c||}{} & \multicolumn{1}{c|}{{\bf 1}} & \multicolumn{1}{c|}{{\bf 2}} & \multicolumn{1}{c|}{{\bf 3}} \\ \cline{2-5} 
				\multicolumn{1}{c|}{\multirow{10}{*}{\begin{sideways}${\bf m/n\in }$\end{sideways}}} & \multicolumn{1}{c||}{${\bf  (0,k_0) \cup (k_0,k_1)}$} & \multicolumn{1}{c|}{4} & \multicolumn{1}{c|}{$6 \leq m_2 \leq 11$} & \multicolumn{1}{c|}{$6 \leq m_3 \leq 11$}  \\ \cline{2-5}
				
				\multicolumn{1}{c|}{} & \multicolumn{1}{c||}{${\bf   \left(k_1,\frac{1}{5} \right)\cup \left(\frac{1}{3},\frac{1}{2} \right)}$} & \multicolumn{1}{c|}{4} & \multicolumn{1}{c|}{7} & \multicolumn{1}{c|}{7} \\ \cline{2-5} 
				
				\multicolumn{1}{c|}{} & \multicolumn{1}{c||}{${\bf   \left(\frac{1}{5},k_2 \right) \cup \left(\frac{3}{2},k_3 \right) }$} & \multicolumn{1}{c|}{4} & \multicolumn{1}{c|}{6} & \multicolumn{1}{c|}{6} \\ \cline{2-5} 
				
				\multicolumn{1}{c|}{} & \multicolumn{1}{c||}{${\bf   \left(k_2,\frac{1}{3} \right) \cup (k_3,2) }$} & \multicolumn{1}{c|}{4} & \multicolumn{1}{c|}{$7 \leq m_2 \leq 8$} & \multicolumn{1}{c|}{$7 \leq m_3 \leq 8$} \\ \cline{2-5}
				
				\multicolumn{1}{c|}{} & \multicolumn{1}{c||}{${\bf   \left(\frac{1}{2},\frac{2}{3} \right) \cup \left(\frac{3}{4},\frac{4}{5} \right) \cup \left(\frac{4}{3},\frac{3}{2} \right) \cup (k_5, \infty)}$} & \multicolumn{1}{c|}{3} & \multicolumn{1}{c|}{7} & \multicolumn{1}{c|}{7} \\ \cline{2-5} 
				
				\multicolumn{1}{c|}{} & \multicolumn{1}{c||}{${\bf  \{\frac{2}{3}\}}$} & \multicolumn{1}{c|}{3} & \multicolumn{1}{c|}{5} & \multicolumn{1}{c|}{5} \\ \cline{2-5} 
				
				\multicolumn{1}{c|}{} & \multicolumn{1}{c||}{${\bf   \left(\frac{2}{3},\frac{3}{4} \right) \cup \left[\frac{4}{5},1 \right) \cup \left(1,\frac{4}{3} \right] }$} & \multicolumn{1}{c|}{3} & \multicolumn{1}{c|}{6} & \multicolumn{1}{c|}{6} \\ \cline{2-5}

				\multicolumn{1}{c|}{} & \multicolumn{1}{c||}{${\bf  \{2\}}$} & \multicolumn{1}{c|}{3} & \multicolumn{1}{c|}{4} & \multicolumn{1}{c|}{4} \\ \cline{2-5} 
				
				\multicolumn{1}{c|}{} & \multicolumn{1}{c||}{${\bf   (2,3)}$} & \multicolumn{1}{c|}{3} & \multicolumn{1}{c|}{$7 \leq m_2 \leq 8$} & \multicolumn{1}{c|}{$7 \leq m_3 \leq 8$} \\ \cline{2-5} 
				
				\multicolumn{1}{c|}{} & \multicolumn{1}{c||}{${\bf   (3,k_4)}$} & \multicolumn{1}{c|}{3} & \multicolumn{1}{c|}{$6 \leq m_2 \leq 11$} & \multicolumn{1}{c|}{$6 \leq m_3 \leq 11$} \\ \cline{2-5} 
				
				\multicolumn{1}{c|}{} & \multicolumn{1}{c||}{${\bf   (k_4,k_5)}$} & \multicolumn{1}{c|}{3} & \multicolumn{1}{c|}{$7 \leq m_2 \leq 9$} & \multicolumn{1}{c|}{$7 \leq m_3 \leq 9$} \\ \cline{2-5} 
			\end{tabular}
		\end{center} 
		\bigskip
		\caption{Values for $m_\ell(m,n)$, with $\ell\in\{1,2,3\}$, when $m$ is even and $n$ is an odd positive integers} \label{tab4}
	\end{table}

We remark that by applying the  linear change of coordinates $(x,y)\rightarrow (y,x)$, it follows that $m_i(m,n)=m_i(n,m)$ for $i,m,n\in \N$. For this reason, we omit the case when $m$ is an odd and $n$ is an even positive integers.

\section{Melnikov Functions}

Consider an open subset $D\subset\R^d$, $d\in\N$, $\mathbb{S}^1=\R/T$ for some period $T>0$, and $N$ a positive integer. Let $\theta_i:D\to \mathbb{S}^1$, $i=1,\ldots,M$, be functions such that $\theta_0(x)\equiv0<\theta_1(x)<\cdots<\theta_M(x)<T\equiv\theta_{M+1}(x)$, for all $x\in D$. Under the assumptions above, we consider the following piecewise smooth differential system 
\begin{equation}\label{general-system2} 
\dot{x} = \sum_{i=1}^N\e^iF_i(t,x)+\e^{N+1}R(t,x,\e),
\end{equation}
with, for $i=1,\ldots,N$, 
\[
F_i(t,x)=\left\{
\begin{array}{ll}
F_i^0(t,x), & 0<t<\theta_1(x), \\
F_i^1(t,x), & \theta_1(x)<t<\theta_2(x), \\
\vdots & \\
F_i^M(t,x), & \theta_M(x)<t<T, 
\end{array}
\right.
\]
and
\[
R(t,x,\e)=\left\{
\begin{array}{ll}
R^0(t,x,\e), & 0<t<\theta_1(x), \\
R^1(t,x,\e), & \theta_1(x)<t<\theta_2(x), \\
\vdots & \\
R^M(t,x,\e), & \theta_M(x)<t<T, 
\end{array}
\right.
\]
also $F_i^j:\mathbb{S}^1\times D\rightarrow\R^d$, $R^j:\mathbb{S}^1\times D\times (-\e_0,\e_0)\rightarrow\R^d$, $i=1,\ldots,N$ and $j=1,\ldots,M$, are $\mathcal{C}^r$ functions, $r\geq N+1$, and $T-$periodic in the variable $t$. In this case, the commutation manifold is given by $\s=\{(\theta_i(x),x);\ x\in D,\ i=0,1,\ldots,M+1 \}$.

In \cite{AndCesCruNov2021} it is shown that the Melnikov Functions  of order $i$ of the system \eqref{general-system2}  is given by
\begin{equation*}
\Delta_i(x)=\dfrac{1}{i!}z_i^M(T,x),
\end{equation*}
where $z_i^j(t,x)$ is defined recurrently for $i=1,\dots,k$ and $j=0,\dots,M$ as follows:
\begin{equation*}
\begin{array}{rl}
z^0_1(t,x)=&\displaystyle\int_0^{t}F_1^0(s,x)ds,\\
z^j_1(t,x)=&z^{j-1}_1(\theta_j(x),x)+\displaystyle\int_{\theta_j(x)}^{t}F_1^j(s,x)ds,\\
z_i^0(t,x)=&i!\displaystyle\int_{0}^t\left( F_{i}^0(s,x)+\displaystyle\sum_{l=1}^{i-1}\displaystyle\sum_{b\in S_l}\dfrac{1}{b_1!b_2!2!^{b_2}\dots b_l!l!^{b_l}}\partial^{L_b}_xF_{i-l}^0(s,x)\displaystyle\prod_{m=1}^l\left(z_m^0(s,x)\right)^{b_m}\right) ds,\\
z_i^j(t,x)=&z_i^{j-1}(\theta_j(x),x)\\
&+i!\displaystyle\int_{\theta_j(x)}^t\left( F_{i}^j(s,x)+\displaystyle\sum_{l=1}^{i-1}\displaystyle\sum_{b\in S_l}\dfrac{1}{b_1!b_2!2!^{b_2}\dots b_l!l!^{b_l}}\partial^{L_b}_xF_{i-l}^j(s,x)\displaystyle\prod_{m=1}^l\left(z_m^j(s,x)\right)^{b_m}\right) ds\\
&+\displaystyle\sum_{p=1}^{i-1}\dfrac{1}{p!}\dfrac{\partial^{p}}{\partial\varepsilon^{p}}\left(\delta_{i-p}^j\left( A_j^p(x,\varepsilon),x\right)  \right) \Big|_{\varepsilon=0},
\end{array}
\end{equation*}
where $\delta_i^j(t,x)=z_i^{j-1}(t,x)-z_i^j(t,x)$ and $A^{p}_{j}(x,\varepsilon)=\displaystyle\sum_{q=0}^p\dfrac{\varepsilon^q}{q!}\alpha_j^q(x)$ with
\begin{equation*}
\alpha_j^q(x)=\displaystyle\sum_{l=1}^q\dfrac{q!}{l!}\displaystyle\sum_{u\in S_{q,l}}D^l\theta_j(x)\left(\displaystyle\prod_{r=1}^l w_{u_r}^j(x)\right),  \text{ for } q=1,..,k,
\end{equation*}
and
\begin{equation*}
\begin{array}{rl}
w_1^j(x)=&z_{1}^{j-1}(\theta_j(x),x),\\
w_i^j(x)=&\dfrac{1}{i!}z_{i}^{j-1}(\theta_j(x),x)\\
&+\displaystyle\sum_{a=1}^{i-1}\displaystyle\sum_{b\in S_a}\dfrac{1}{(i-a)!b_1!b_2!2!^{b_2}\dots b_a!a!^{b_a}}\partial_t^{L_b}z_{i-a}^{j-1}(\theta_j(x),x)\displaystyle\prod_{m=1}^a\left(\alpha_j^m(x) \right)^{b_m}.
\end{array}
\end{equation*}
Here  $\partial_x^{L_b}G(t, x)$ denotes the derivative of order $L_b$ of a function G, with respect to the variable $x$, $S_a$ is the set of all a-tuples of non-negative integers $(b_1, b_2, \dots , b_a)$ satisfying $b_1 + 2b_2 + \dots +
ab_a = a$, $L = b_1 + b_2 + \dots + b_a$, and $S_{q,a}$  is the set of all a-tuples of positives integers $(b_1, b_2, \dots , b_a)$ satisfying $b_1 + b_2 + \dots +
b_a = q$.
In \cite{AndCesCruNov2021} the following result concerning to the Melnikov functions is proved.
\begin{mtheorem}\label{thm:melnikov}
	Consider the nonsmooth differential system \eqref{general-system2} and denote $\Delta_0=0.$ Assume that, for some $\ell\in\{1,\ldots,k\},$ $\Delta_i=0$, for $i=1,\dots,\ell-1,$ and $\Delta_{\ell}\neq0$. If  $\Delta_{\ell}(a^*)=0$ and $\det(J\Delta_{\ell}(a^*))\neq0,$ for some $a^*\in D,$ then, for $|\e|\neq0$ sufficiently small, there exists a unique $T$-periodic solution $x(t,\e)$ of system \eqref{general-system2} such that $x(0,\e)\to a^*$ as $\e\to0$.  
\end{mtheorem}

\section{Chebyshev Systems}\label{sec:chebyshev}

We recall the definition of the Wronskian of orded set  $[u_0, \ldots, u_{s}]$, of $s+1$ :
\begin{equation*}\label{wrons}
W(x)= W(u_0,\ldots, u_s)(x)=\det(M(u_0, \ldots, u_s)(x)),
\end{equation*}
with
\[
M(u_0, \ldots, u_s)(x)=\left( \begin{array}{ccc}
u_0(x)& \ldots& u_s(x)\\
u'_0(x)& \ldots & u'_s(x)\\
\vdots& & \vdots\\
u_0^{(s)}(x) &&  u_s^{(s)}(x) 
\end{array}\right).
\]

Let $\mathcal{F}=\left[u_{0}, \ldots, u_{n}\right]$ be an ordered set of smooth functions defined on the closed interval $[a, b]$. We denote by \(\operatorname{Span}(\mathcal{F})\)  the set of all functions given by linear combinations of  $\F$ and \(\mathcal{Z}(\mathcal{F})\)  the maximum number of isolated zeros, counting multiplicity, of any function in $\operatorname{Span}(\mathcal{F})$. The ordered set  $\CF$ is called an {\it Extended Chebyshev} system or  {\it ECT-system} on $[a,b]$ if $\mathcal{Z}(\CF)= n$. Besides   $\CF$ is called an ET-system with {\it accuracy }$k$ on $[a,b],$  if $\mathcal{Z}(\CF)= n+k$, (see for more details \cite{KASTU1966, NOTOR2017}).\\

A classical result concerning ECT-system is the following:
\begin{theorem}[\cite{KASTU1966}]\label{t1}
Let $\F=[u_0, u_1, ..., u_n]$ be an ECT-system on a closed interval  $[a, b]$. Then, the number of isolated zeros for every element of $\mbox{Span}(\F)$ does not exceed $n$. Moreover, for each configuration of $m \leq n$ zeros, taking into account their multiplicity, there exists $F\in \mbox{Span}(\F)$ with this configuration of zeros.
\end{theorem}
The next results, proved in \cite{NOTOR2017} and \cite{AndCesCruNov2021}, extend the results of Theorem \ref{t1} for  ordered sets of smooth functions  when some
of the Wronskians vanish

\begin{theorem}[\cite{NOTOR2017}]\label{t2}
	Let $\F=[u_0, u_1, \ldots u_n]$ be an ordered set of functions on $[a, b]$. Assume that all the Wronskians $W_i(x)$, $i\in\{0, \ldots, n-1\}$, are nonvanishing except $W_n(x)$, which has exactly one zero on $(a, b)$ and this zero is simple. Then, the number of isolated zeros for every element of $Span(\F)$ does not exceed $n+1$. Moreover, for any configuration of $m \leq n+1$ zeros there exists $F\in Span(\F)$ realizing it.
\end{theorem}

\begin{theorem}[\cite{AndCesCruNov2021}]\label{t3} 
 Let  $\mathcal{F}=\left[u_{0}, u_{1}, \ldots, u_{n}\right]$   be an ordered set of analytic functions in $[a, b].$  Assume that, for each $i\in\{0, \ldots, n\},$ the Wronskian  $W_{i}$ has $\nu_{i}$ zeros counting multiplicity. Then, the number of simple zeros for every element  of  $\operatorname{Span}(\mathcal{F})$  does not exceed
\begin{equation}\label{upper}
n+\nu_{n}+\nu_{n-1}+2\left(\nu_{n-2}+\cdots+\nu_{0}\right)+\mu_{n-1}+\cdots+\mu_{3},
\end{equation}
where $\mu_i=\min(2 \nu_i, \nu_{i-3}+\ldots+\nu_0)$, for $i\in\{3, \ldots, n-1\}.$
\end{theorem}

\subsection{New families of ET-systems with accuracy} We define on $[a, b]$, with $a, b>0$, the ordered set of functions 

\begin{eqnarray}
	\G_0^k &=& [u_{1}^k], \nonumber \\
	\G_1^k &=& [u_0^k,u_{1}^k], \nonumber \\
	\G_2^k &=& [u_{8}^k,u_{11}^k,u_{1}^k], \nonumber \\
	\G_3^k &=& [u_0^k,u_2^k,u_{8}^k], \nonumber \\
	\G_4^k &=& [u_{11}^k,u_{8}^k, u_1^k,u_2^k], \nonumber \\
	\G_5^k &=& [u_2^k,u_5^k,u_{4}^k, u_0^k,u_1^k,u_6^k,u_7^k,u_{9}^k], \nonumber \\
 \G_{6}^k&=&[u_{6}^k,u_{0}^k,u_{1}^k,u_{10}^k,u_{12}^k] \nonumber\\
\G_7^k&=&[u_{5}^k, u_{6}^k,u_{3}^k,u_{1}^k,u_{10}^k,u_{12}^k] \nonumber\\
\G_8^k&=&[u_{1}^{k}, u_{6}^{k}, u_{10}^{k}, u_{12}^{k}, u_{3}^{k}] \nonumber\\
	\mathcal{G}_{9}^{k} &=& [u_0^{k}, u_{1}^{k}, u_{6}^{k}, u_{5}^{k}, u_{3}^{k}, u_{10}^{k}, u_{12}^{k}], \nonumber \\
	\G_{10}^k &=& [u_{0}^k,u_{1}^k,u_{13}^k], \nonumber \\
	\mathcal{G}_{11}^{k} &=& [u_1^{k}, u_{3}^{k}], \nonumber 
\end{eqnarray}
with $k\in \R$ where
\[
\begin{array}{ll}
u_{0}^k(x)=1, &  u_1^k(x)=x^k,\\
u_2^k(x)= x^{k-1},& u_3^k(x)=x^{k-2},\\
u_4^k(x)= x^{2k-1},&  u_5^k(x)= x^{2k-2},\\
u_6^k(x)=x^{3k-2},&u_7^k(x)=x^{3k-3},\\
u_8^k(x)= x+x^{2k-1},&u_{9}^k(x)= x+kx^{4k-3},\\
u_{10}^{k}(x)= \left(1+x^{2k-2}\right) \left( x+kx^{2k-1}\right),& u_{11}^k(x)= \left(x+x^{2k-1}\right) \tan ^{-1}\left(x^{k-1}\right), \\
u_{12}^{k}(x)=\left(1+x^{2k-2}\right) \left( x+kx^{2k-1}\right) \tan ^{-1}\left(x^{k-1}\right) ,&u_{13}^k(x)=x^{2k}. \end{array}
\]
 In the following, we denote by $\sgn(a_1, a_2, \ldots, a_n)$ the vector   $(\sgn(a_1), \sgn(a_2), \ldots, \sgn(a_n))$  for $(a_1, a_2, \ldots, a_n)\in \R^n$.
\begin{proposition} \label{prop5}The following statements hold
\begin{enumerate}
\item $\z(\G_0^k)=0$,  for $k>0$; 
\item $\z(\G_1^k)=1$, for $k>0$;
\item  $\z(\G_2^k)=2$, for $k\ne1$;
\item $\z(\G_3^k)=2$, for $k\in(1,2]$;
\item  $\z(\G_3^k)=3$, for $k\in(2,\infty)$;
\item $\z(\G_4^k)=3$, for $k\in  \left(1,\frac{3}{2} \right] \cup [2, \infty)$;
\item  $\z(\G
_4^k)=4$, for $k\in  \left(0,\frac{1}{2} \right) \cup \left(\frac{3}{2}, 2 \right)$;
\item $\z(\G_5^k)=7$, for $k \in \left[\frac{4}{3},\frac{3}{2} \right) \cup (2, \infty)$;
\item  $\z(\G_5^k)=8$, for  $k\in \left(1,\frac{4}{3} \right) \cup \left(\frac{3}{2},2 \right)$;
\item  $\z(\G_6^k)=4$, for  $k=2$;
\item  $\z(\G_7^k)=5$, for  $k=\frac{2}{3}$;
\item $\z(\G_{10}^k)=2$, for $k>0$;
\item $\z(\G_{11}^k)=1$, for $k=1$.
\end{enumerate}
\end{proposition} 
\begin{proof}

The proof of  statements \textit{(1), (2), (12), (13) } are immediate and we omit it. To prove the statements \textit{(3)–(11)} we will show that for each family $\G\in \{\G_2^k,\G_3^k, \G_4^k, \G_5^k, \G_6^k, \G_7^k \}$ all its Wronskians are nonvanishing except the last, which has no zeros or one zero,  which is simple. Therefore, by Theorems \ref{t1} and \ref{t2} we will get that $\G$ is an ECT-System ou an ET-System with accuracy one and consequently $\z(\G)=n$ or $\z(\G)=n+1$, where $n$ is the cardinal of $\G.$\\
\bigskip

The Wronskians of the family $\G_2^k$ are given by
\begin{eqnarray}
	W_0(x) &=& x+x^{2k-1}, \nonumber \\
	W_1(x) &=& (k -1) x^{k-2 } \left(x^{2 }+x^{2k}\right), \nonumber \\
	W_2(x) &=& -\frac{4 (k -1)^3 x^{4k-2}}{x^{2}+x^{2k}}. \nonumber
\end{eqnarray}
Clearly $W_0(x), W_1(x), W_2(x)\neq 0$ in $(0, \infty)$ for $k>0.$  

\bigskip
The Wronskians of the family $\G_3^k$ are given by
	\[
	\begin{array}{rcl}
	W_0(x) &=& 1, \nonumber \\
          W_1(x) &=& (k-1) x^{k-2}, \nonumber \\
          W_2(x) &=& (k-1) x^{k-3} P_{0,k}(x^{2k-2}), \nonumber
	\end{array}
	\]	
with $P_{0,k}(x)= (k  (2k-1))x-(k-2)$. For $k>1$,  $W_0(x), W_1(x)$ do not vanish in $(0, \infty)$. Notice that $P_{0,k}(0)P'_{0,k}(0)>1$ if $k \in (1,2]$ and   $P_{0,k}(0)P'_{0,k}(0)<1$ if $k \in (2,\infty)$. Thus, $W_2(x)\neq 0$ in $(0, \infty)$ for $k \in (1,2]$ and  $W_2(x)$  has exactly one positive zero, which is simple, for $k \in(2, \infty)$.\\
\bigskip

The Wronskians of the family $\G_4^k$ are given by
	\[
	\begin{array}{rcl}
	W_0(x) &=& \left(x+ x^{2k-1}\right) \tan ^{-1}\left(x^{k-1}\right), \nonumber \\
W_1(x) &=& -(k-1) x^{k-2} \left(x^{2k}+x^2\right), \nonumber \\
W_2(x) &=& \dfrac{4 (k -1)^3 x^{4k-2}}{x^{2k}+x^2}, \nonumber \\
W_3(x) &=& \dfrac{4 (k -1)^3 x^{5k-4}}{\left(x^{2k }+x^2\right)^2}P_{1,k}(x^{2k-2}), \nonumber
	\end{array}
	\]	
with $P_{1,k}(x)= ( k (2 k -3))x+( (k -2) (2k -1))$. It is easy to see that $W_0(x), W_1(x), W_2(x) \neq 0$ for $x>0.$ Further, $||(P_{1,k}(0), P'_{1,k}(0))||\neq 0$,  $P_{1,k}(0)P'_{1,k}(0)\geq 0$ for $k \in \left(1,\frac{3}{2} \right] \cup [2, \infty)$ and   $P_{1,k}(0)P'_{1,k}(0)< 0$ for $k \in \left(\frac{3}{2}, 2 \right)$. Hence, $W_3(x)\neq0$  in $(0, \infty)$ for $k \in \left(1,\frac{3}{2} \right] \cup [2, \infty)$ and  $W_3(x)$  has exactly one positive zero, which is simple, if $k \in \left(\frac{3}{2}, 2 \right)$.\\
\bigskip

The Wronskians of the family $\G_5^k$ are given by
\begin{eqnarray}
W_0(x) &=& x^{k-1}, \nonumber \\
W_1(x) &=& (k-1) x^{3k-4}, \nonumber \\
W_2(x) &=& (k-1)k  x^{5k-7}, \nonumber \\
W_3(x) &=& -2 (k-1)^3 k (2k -1) x^{5k-10}, \nonumber \\
W_4(x) &=& -2 (k-2) (k-1)^4 k^2 (2k-1) x^{6k-14}, \nonumber \\
W_5(x) &=& -4 (1-2k)^2 (k-2) (k-1)^6 k^3 (3k-2) x^{9k-21}, \nonumber \\
W_6(x) &=& 24 (1-2k)^2 (k-2)^2 (k-1)^9 k^3 (2k-3)(3k-2) x^{12k-30}, \nonumber \\
W_7(x) &=& 144 (1-2k)^2 (k-2)^2 (k-1)^{12} k^3 (2k-3)(3k-2) x^{12k-36} P_{2,k}(x^{4k-4}), \nonumber
\end{eqnarray}
with $P_{2,k}(x)=(k^2 (2k-1) (3k-2) (4k-3))x+( (k-2) (2k-3) (3k-4))$.
For $k\in(1,\infty) \setminus\{2\}$ and $i\in\{0,1,2,3,4,5,6\}$, $W_i(x) \neq 0$ in $(0, \infty)$. Analogously to the previous case, we conclude that  $W_7(x)\neq(0, \infty)$ for $k \in [\frac{4}{3},\frac{3}{2}) \cup (2, \infty)$ and  $W_7(x)$  has exactly one positive zero, which is simple, for  $k\in(1,\frac{4}{3}) \cup (\frac{3}{2},2)$.
\bigskip

The Wronskians of  $\G_6^k$, for $k=2$, are given by
\begin{eqnarray}
W_0(x) &=& x^4, \nonumber \\
W_1(x) &=& -4 x^3, \nonumber \\
W_2(x) &=& 16 x^3, \nonumber \\
W_3(x) &=& 48 \left(10 x^5-3 x^3+x\right),\nonumber \\
W_4(x) &=& \frac{1536 x^3 \left(2 x^2+9\right)}{\left(x^2+1\right)^3}. \nonumber 
\end{eqnarray}
By straightforward computations, we obtain that $W_i(x) \neq 0$, $i \in \{0,1,2,3,4\}$.\\
\bigskip

The Wronskians of  $\G_7^k$, for $k=\frac{2}{3}$, are given by
\begin{eqnarray}
W_0(x) &=& \frac{1}{x^{2/3}}, \nonumber \\
W_1(x) &=& \frac{2}{3 x^{5/3}}, \nonumber \\
W_2(x) &=& \frac{16}{27 x^5}, \nonumber \\
W_3(x) &=& \frac{256}{243 x^{22/3}},\nonumber \\
W_4(x) &=& \frac{256 \left(-25 x^{2/3}+105 x^{4/3}+6\right)}{19683 x^{35/3}}, \nonumber \\
W_5(x) &=& -\frac{65536 \left(9 x^{2/3}+26\right)}{4782969 \left(x^{2/3}+1\right)^4 x^{46/3}}. \nonumber
\end{eqnarray}
Notice that $W_i (x)\neq 0$, $i \in \{0,1,2,3,4,5\}$.

\end{proof}

\begin{proposition} \label{prop6}The following statements hold.
\begin{enumerate}
\item $\z(\G_8^k)=4$, for $k\in(1, 2] \cup [3,5]$;
\item$\z(\G_8^k)=5$, for  $k\in(2,3)\cup (5, \infty).$
\end{enumerate}
\end{proposition} 

\begin{proof}

We consider the ordered set
\[
\F_1^k=\left\lbrace 
\begin{array}{ll}
\left[u_{1}^{k}, u_{6}^{k}, u_{10}^{k}, u_{12}^{k}, u_{3}^{k}\right],& k\in (1,5),\\
\left[u_{1}^{k}, u_{3}^{k}, u_{10}^{k}, u_{12}^{k},u_{6}^{k}\right],& k \in  [5,\infty). \\
\end{array}\right. 
\]
Observe that $\z(\G_8^k)=\z(\F_1^k)$.  In the sequel, we will show that for $k \in [1, 5)$  all  Wronskians $\F_1^k$ are nonvanishing except the last, which has no zeros or  exactly one zero, which is simple. Therefore, by Theorems \ref{t1} and \ref{t2} we will get that $\F_1^k$ is an ECT-System or an ET-System with accuracy one and consequently $\z(\G_8^k)=4$ or $\z(\G_8^k)=5$. In order to study  $\z(\F_1^k)$  we consider the following two cases: 
\subsection*{Case 1:} Let $k \in (1,5)$. The Wronskians of $\F^k_1$ are given by  
\begin{eqnarray}
W_0(x) &=& x^{k}, \nonumber \\
W_1(x) &=&2 (k-1) x^{4k-3}, \nonumber \\
W_2(x) &=& 2 (k-1)^3 x^{4k-4} P_{3,k}(x^{2k-2}), \nonumber \\
W_3(x) &=& -\frac{32 (k-1)^6 x^{9k-7}}{\left(x^{2k}+x^2\right)^2} P_{4,k}(x^{2k-2}),\nonumber \\
W_4(x) &=& \frac{256 (k-1)^6k x^{10k-11}}{\left(x^{2k}+x^2\right)^3} P_{5,k}(x^{2k-2}), \nonumber 
\end{eqnarray}
where $P_{3,k}(x)=3 kx^{2} -(k+1) x+3$, $P_{4,k}(x)= (k-5)k x+(1-5k)$ and $$P_{5,k}(x)=A_{5,k} x^{2}+B_{5,k}x+C_{5,k}$$
with
\[
\begin{array}{ll}
A_{5,k}= (k-5) (k-2) k (3k -1), &B_{5,k}=-2 (k (k(k(k+10)-30)+10)+1),\\
 C_{5,k}= -(k-3) (2k-1) (5k-1). \end{array}
\]
From here, it is easy to see that for $k\in (1,5)$, $i \in\{0,1,3\}$, $W_i(x) \neq 0$ in $(0, \infty)$ and $\Dis(P_{3,k})= k^2-34k+1<0$. Thus, $W_2(x)$ has no positive zeros. Moreover,

\begin{equation*}
\sgn(A_{5,k},B_{5,k}, C_{5,k})=\left\{
\begin{array}{ll}
(1,1,1),&  k\in(1,2),  \nonumber \\
(0,1,1),& k=2, \nonumber \\
(-1,\sgn(B_{5,k}),1),& k\in(2,3), \nonumber \\
(-1,-1,0),& k=3,\nonumber \\
(-1,-1,-1),& k\in(3,5). \nonumber 
\end{array}\right.
\end{equation*}

 Therefore, $W_4(x)$ has no zeros for $k\in(1, 2] \cup [3,5)$ and has one positive simple zero for $k\in(2,3)$.
\subsection*{Case 2:} Let $k\in[5, \infty)$. The Wronskians of $\F^k_1$ are given by  
\begin{eqnarray}
W_0(x) &=& x^{k}, \nonumber \\
W_1(x) &=&-2 x^{2 k-3}, \nonumber \\
W_2(x) &=& -2 (k-1) x^{2k-4} P_{6,k}(x^{-2+2k}), \nonumber \\
W_3(x) &=& \frac{8 (k -1)^3 x^{5k-5}}{\left(x^{2k }+x^2\right)^2} P_{7,k}(x^{2k-2}),\nonumber \\
W_4(x) &=& \frac{256 (k-1)^6kx^{10 k-11}}{\left(x^{2k}+x^2\right)^3} P_{8,k}(x^{2k-2}), \nonumber 
\end{eqnarray}
with
\begin{equation*}
\begin{array}{lll}
P_{6,k}(x)&=& A_{6,k} x^{2}+ B_{6,k}x+C_{6,k}, \\
	P_{7,k}(x)&=& A_{7,k} x^{4}+B_{7,k} x^{3}+C_{7,k} x^{2}+D_{7,k} x+E_{7,k}, \\
		P_{8,k}(x)&=&A_{8,k} x^{2}+B_{8,k}x+C_{8,k},
\end{array}
\end{equation*}
and

\[
\begin{array}{ll}
A_{6,k}= 3 k(3 k -1),& B_{6,k}= (k +1)^2,\\
C_{6,k}=k -3,& A_{7,k}=-3k^3 (3k-1),\\
B_{7,k}= -17k^4+8 k ^3+k^2, &  C_{7,k}= -k(k(k(k+34)-67)+20),  \\
D_{7,k}=(4-k(k(k(k+22)-61)+18)),& E_{7,k}=k  (k(7-3k)+6), \\	A_{8,k}= -(k-5) (k-2) k (3k -1),& B_{8,k}=2 (k (k(k(k+10)-30)+10)+1), \\
C_{8,k}= (k-3) (2k-1) (5k-1).
 \end{array}
\]
Notice that for $k\in [5, \infty)$,
$\sgn(A_{6,k})=\sgn(B_{6,k})=\sgn(C_{6,k})=1$ and
\[\sgn(A_{7,k})=\sgn(B_{7,k})=\sgn(C_{7,k})=\sgn(D_{7,k})=\sgn(E_{7,k})=-1.
\]
Thus, $W_i(x)\neq 0$ in $(0, \infty)$ for $k \in[5, \infty)$ and  $i \in\{0,1,2,3\}$.  Since
\begin{equation}
\sgn(A_{8,k}, B_{8,k}, C_{8,k})=\left\{
\begin{array}{ll}
(1,-1,-1),&  k\in(5,\infty),  \nonumber \\
(0,-1,-1),& k=5, \nonumber \\
\end{array}\right.
\end{equation}
we get that $W_4 (x)\neq0$ in $(0, \infty)$ for $k=5$ and $W_4 (x)$  has one positive zero, which is simple, for $k\in(5,\infty).$ 
\bigskip
 This ends the proof of Proposition \ref{prop6}.
\end{proof}

\begin{proposition} \label{prop8}
 If $k \in (0, 5]\setminus \{k_0, k_1, k_2, k_3, k_4, k_5, 2,1, \frac{3}{2}\}$, then $6\leq \z(\G_9^k)\leq 11$. Moreover, the following statements hold
\begin{enumerate} 
\item $\z(\G_9^k)=6$, for $k \in \left(\frac{1}{5}, k_2 \right) \cup \left(\frac{2}{3},\frac{3}{4} \right) \cup \left[\frac{4}{5},1 \right) \cup \left(1,\frac{4}{3} \right] \cup \left(\frac{3}{2}, k_3 \right)$;
\item $\z(\G_9^k)=7$,  for $k \in \left(k_1,\frac{1}{5} \right) \cup \left(\frac{1}{3},\frac{1}{2} \right) \cup \left(\frac{1}{2},\frac{2}{3}\right) \cup \left(\frac{3}{4},\frac{4}{5}\right) \cup \left(\frac{4}{3},\frac{3}{2}\right) \cup (k_5, 5]$;
\item $7\leq \z(\G_9^k)\leq8$ for $k \in \left(k_2,\frac{1}{3} \right) \cup (k_3,2) \cup (2,3)$;
\item $7\leq \z(\G_9^k)\leq9$, for $k \in (k_4,k_5)$,
\end{enumerate}
where \begin{enumerate}[(i)]
	 	\item  $k_1$ is the irrational root in $\left(0, \frac{1}{5} \right)$ of  $$q_1(x)=2 x^6-153 x^5+894 x^4-1625 x^3+1234 x^2-444 x+48;$$
	 	\item $k_2,k_3,k_5$ are the irrational roots in $\left(\frac{1}{5},\frac{1}{3} \right),\left(\frac{3}{2},2 \right),$ and  $(3,4)$, respectively, of  
  		\[
  		\begin{array}{ll}
  		q_2(x)=&177586560 x^{17}-1447424208 x^{16}+5969663136 x^{15}-29387373904 x^{14}\\&+129626832188 x^{13}-293774330511 x^{12}+102470736381 x^{11}+1027573184492 x^{10}\\&-2645532232771 x^9+3259827826136 x^8-2344796073539 x^7+997977148820 x^6\\&-227233713561 x^5+15757275163 x^4+3311923726 x^3-563207524 x^2\\&+11249208 x+19744;
  		\end{array}
  		\]
	 	
	 	\item  $k_4$  is the irrational root in $(3,4)$ of $$q_3(x)=862 x^5-4831 x^4+8308 x^3-5186 x^2+974 x+1.$$
 \end{enumerate}

\end{proposition} 
In order to prove Proposition \ref{prop8}, we need the following lemma whose proof is reported in Appendix \ref{Apend}.
\begin{lemma} \label{prop.pol}
	Let $\mathcal{P}(x)= \mathcal{A} x^{4}+\mathcal{B} x^{3}+\mathcal{C}x^{2}+ \mathcal{D}x+\mathcal{E}$ 
 with
	\[
	\begin{array}{l}
	\mathcal{A}=(k-5) (k-2)^2 k (2 k-1) (3 k-4) (3 k-1) (4 k-3),\\
	\mathcal{B}=-2 (2 k-1) (3 k-4) (k (k (k (2 k (k (39 k-179)+235)-89)-118)+35)-2),\\  \mathcal{C}= (3 k-4) (k (k (k (5 k (k (2 k (24 k-61)+177)-366)+2034)-930)+201)-38),\\
	\mathcal{D}= -2 (3 k-2) (5 k-4) (k (k (k (k (2 k-19)+88)-75)-38)+26) ,  \\
	\mathcal{E}=(k-3) (2 k-3) (2 k-1) (3 k-2) (5 k-4) (5 k-1).
	\end{array}
	\]
	 The following statements hold.

	\begin{enumerate}
		\item For $k \in \left(\frac{1}{5}, k_2 \right) \cup \left[\frac{2}{3},\frac{3}{4} \right) \cup \left[\frac{4}{5},1 \right) \cup \left(1,\frac{4}{3} \right] \cup \left(\frac{3}{2}, k_3 \right)$, $\mathcal{P}(x)\neq 0$ in $(0, \infty)$, where $k_2,k_3$ are roots of $q_2(x)$ in $\left(\frac{1}{5}, \frac{1}{3} \right)$  and $\left(\frac{3}{2},2 \right)$, respectively;

		\item For $k \in \left(k_0,\frac{1}{5} \right) \cup \left(\frac{1}{3},\frac{1}{2} \right) \cup \left(\frac{1}{2},\frac{2}{3} \right) \cup \left(\frac{3}{4},\frac{4}{5} \right) \cup \left(\frac{4}{3},\frac{3}{2} \right) \cup \{2\} \cup (k_5,5]$, $\mathcal{P}(x)$   has one positive zero, which is simple, where  $k_5$ and $k_0$  are roots of $q_2(x)$ in  $(3,4)$ and $\left(0,\frac{1}{5} \right)$, respectively;
		\item For $k \in \left(k_2,\frac{1}{3} \right) \cup (k_3,2) \cup (2,3) \cup (5, \infty)$, $\mathcal{P}(x)$   has two positive zeros, which are simple;
		\item For  $k \in (0,k_0) \cup (3,k_5)$, $\mathcal{P}(x)$  has three positive zeros, which are simple.
	\end{enumerate}
\end{lemma}

\begin{proof}[Proof of Proposition \ref{prop8}]
Let $k \in \R^+\setminus \left\{k_0, k_1, k_2, k_3, k_4, k_5, 2,1, \frac{3}{2} \right\}$. The Wronskians of the ordered set $\G_9^k$ are given by
\begin{equation}\label{l10}
	\begin{array}{lll}
	W_0(x) &=& 1,\\ 
	W_1(x) &=&k x^{k-1}, \\
	W_2(x) &=& 2 (k-1) k (3 k-2) x^{4 k-5}, \\
	W_3(x) &=& -4 (k-2) (k-1)^2 k^2 (3 k-2) x^{6 k-10},\\
	W_4(x) &=& 16 k^4 (3 k-2) \left(k^2-3 k+2\right)^2 x^{7 k-16},\\
	W_5(x) &=& 16 (k-2)^2 (k-1)^4 k^4 (3 k-2) x^{7 k-20} P_{9,k}(x^{2k-2}),\\
	W_6(x) &=&\frac{512 (k-2)^2 (k-1)^7 k^4 (3 k-2) x^{12 k-20} }{\left(x^{2 k}+x^2\right)^5}\mathcal{P}(x^{2k-2}),\\
	\end{array}
\end{equation}
	with
	\begin{eqnarray}
	P_{9,k}(x)&=& A_{9,k}x^{2}+B_{9,k}x+C_{9,k}, \nonumber 
	\end{eqnarray}
	and
	\[
	\begin{array}{ll}
		A_{9,k}= 3 k (2 k-1) (3 k-1) (4 k-3), \\ B_{9,k}=-(k+1)^2 (2 k-1), \\
		C_{9,k}=3 (k-3) (2 k-3).
	\end{array}
	\]
Clearly $W_i(x)\neq 0$ in $(0, \infty)$ for $i \in\{1,2,3,4\}$,  $\text{degree}(P_{9,k})=2$ and $\text{degree}(\mathcal{P})=4$ for $k \neq \{1,\frac{2}{3},2\}$. This implies that  $W_5(x)$ and $W_6(x)$ have at most 2 and 4 positive zeros, respectively.  Applying Theorems \ref{t1} and \ref{t3}, we conclude that $6\leq\z(\G_9^k)\leq 11$.  However, to determine a better estimate for $\z(\G_9^k)$ using the Chebyshev Theory, it is necessary to rearrange the functions of $\G_9^k$  conveniently. We consider the ordered set $\F_2^k$ defined by
\[
\F_2^k=\left\lbrace 
\begin{array}{ll}
  \left[u_5^{k}, u_{6}^{k}, u_{3}^{k}, u_{0}^{k}, u_{1}^{k}, u_{10}^{k}, u_{12}^{k} \right], & k \in (k_4,k_5) \cup (k_5,4 ], \nonumber\\
\left[u_5^{k}, u_{6}^{k}, u_{3}^{k}, u_{1}^{k}, u_{10}^{k}, u_{12}^{k}, u_{0}^{k} \right], &k \in \left(\frac{3}{4},\frac{4}{5} \right], \nonumber \\
\left[u_6^{k}, u_{3}^{k}, u_{1}^{k}, u_{10}^{k}, u_{12}^{k}, u_{0}^{k}, u_{5}^{k} \right], & k \in \left(\frac{1}{2},\frac{2}{3} \right) \cup \left(\frac{2}{3},\frac{3}{4} \right)\cup \left(\frac{4}{3},\frac{3}{2} \right), \nonumber \\
\left[u_6^{k}, u_{1}^{k}, u_{10}^{k}, u_{12}^{k}, u_{5}^{k}, u_{0}^{k}, u_{3}^{k} \right],& k \in \left(\left(\frac{1}{5},\frac{1}{3} \right)\cup\left(\frac{4}{5},\frac{4}{3} \right] \cup \left(\frac{3}{2},2 \right) \right) \setminus \{k_2,1, k_3\}, \nonumber \\
\left[u_5^{k}, u_{3}^{k}, u_{0}^{k}, u_{1}^{k}, u_{10}^{k}, u_{12}^{k}, u_{6}^{k} \right], & k \in \left(k_1,\frac{1}{5} \right) \cup (2,3), \nonumber \\
\left[u_0^{k}, u_{1}^{k}, u_{6}^{k}, u_{5}^{k}, u_{3}^{k}, u_{10}^{k}, u_{12}^{k} \right], & k \in \left(\frac{1}{3},\frac{1}{2}\right) \cup (4,5]. \nonumber \\

\end{array}\right. 
\]
Observe that $\z(\G_9^k)=\z(\F_2^k)$. In the sequel, we show that, for  $\F_2^k$, all their Wronskians are nonvanishing except the last, which has exactly $\ell$ zeros, which are simple, where
	\[
	\ell=\left\lbrace 
	\begin{array}{ll}
	0, & k \in \left(\frac{1}{5}, k_2\right) \cup \left(\frac{2}{3},\frac{3}{4} \right) \cup \left[\frac{4}{5},1 \right) \cup \left(1,\frac{4}{3} \right] \cup \left(\frac{3}{2}, k_3\right),\\
	1, & k \in \left(k_1,\frac{1}{5} \right) \cup \left(\frac{1}{3},\frac{1}{2} \right) \cup \left(\frac{1}{2},\frac{2}{3}\right) \cup \left(\frac{3}{4},\frac{4}{5}\right) \cup \left(\frac{4}{3},\frac{3}{2} \right) \cup (k_5,5],\\
	2, & k \in \left(k_2,\frac{1}{3}\right) \cup (k_3,2) \cup (2,3),\\
3, & k \in (k_4,k_5).
	\end{array}\right. 
	\]
	From Theorems \ref{t1}, \ref{t2}  and \ref{t3}, we conclude that: 
\begin{enumerate}
\item$\z(\F_2^k)=6$, for $k \in \left(\frac{1}{5}, k_2 \right) \cup \left(\frac{2}{3},\frac{3}{4} \right) \cup \left[\frac{4}{5},1 \right) \cup \left(1,\frac{4}{3} \right] \cup \left(\frac{3}{2}, k_3 \right)$; 
\item $\z(\F_2^k)=7$, for $k \in \left(k_1,\frac{1}{5} \right) \cup \left(\frac{1}{3},\frac{1}{2} \right) \cup \left(\frac{1}{2},\frac{2}{3}\right) \cup \left(\frac{3}{4},\frac{4}{5}\right) \cup \left(\frac{4}{3},\frac{3}{2} \right) \cup (k_5, 5]$;
\item $7 \leq \z(\F_2^k) \leq 8$, for $k \in \left(k_2,\frac{1}{3} \right) \cup (k_3,2) \cup (2,3)$;
\item $7 \leq \z(\F_2^k) \leq 9$, for $k \in(k_4,k_5)$.
\end{enumerate}
In order to study the Wronskians of $\F_{2}^k$ we consider the following six cases:

\subsection*{Case 1: } Let $k\in (k_4,k_5) \cup (k_5,4]$. The Wronskians of $\F^k_{2}$ are given by
	\begin{eqnarray}
	W_0(x) &=& x^{2 k-2}, \nonumber \\
	W_1(x) &=&k x^{5 k-5}, \nonumber \\
	W_2(x) &=& 2 k^3 x^{6 k-9}, \nonumber \\
	W_3(x) &=& -4 (k-2) (k-1) k^3 (3 k-2) x^{6 (k-2)},\nonumber \\
	W_4(x) &=& -16 k^4 (3 k-2) \left(k^2-3 k+2\right)^2 x^{7 k-16}, \nonumber \\
	W_5(x) &=& -16 (k-2)^2 (k-1)^4 k^4 (3 k-2) x^{7 k-20} P_{10,k}(x^{2k-2}), \nonumber \\
	W_6(x) &=& -\frac{512 (k-2)^2 (k-1)^7 k^4 (3 k-2) x^{12 k-20}}{\left(x^{2 k}+x^2\right)^5} \mathcal{P}(x^{2k-2}), \nonumber
\end{eqnarray}
where $$P_{10,k}(x) = (3 k (2 k-1) (3 k-1) (4 k-3))x^{2} -(k+1)^2 (2 k-1) x+3 (k-3) (2 k-3).$$ The discriminant of $P_{10,k}$ is given by
	\begin{equation*}
	\Dis(P_{10,k})=-(2 k-1) \left(862 k^5-4831 k^4+8308 k^3-5186 k^2+974 k+1\right).
	\end{equation*}	
Through straightforward computation, it is direct to see that $\Dis(P_{10,k})<0$. Hence $W_i(x)\neq0$ in $(0, \infty)$ for $i\in\{0,1,2,3,4,5\}$. By Lemma \ref{prop.pol}, we get that  $W_6(x)$ has one and three positive zeros, which are simple, for $k \in (k_5,4]$  and $k \in (k_4,k_5) $, respectively.

\subsection*{Case 2:} Let $k \in \left(\frac{3}{4},\frac{4}{5} \right]$. Then, the Wronskians of $\F^k_{2}$ are given by
	\begin{eqnarray}
		W_0(x) &=& x^{2 k-2}, \nonumber \\
		W_1(x) &=&k x^{5 k-5}, \nonumber \\
		W_2(x) &=& 2 k^3 x^{6 k-9}, \nonumber \\
		W_3(x) &=& 8 (k-2) (k-1) k^3 x^{7 k-12},\nonumber \\
		W_4(x) &=& 8 (k-2) (k-1)^3 k^3 x^{7 k-15} P_{11,k}(x^{2k-2}), \nonumber \\
		W_5(x) &=& \frac{256 (k-2) (k-1)^6 k^3 x^{12 k-16}}{\left(x^{2 k}+x^2\right)^4} P_{12,k}(x^{2k-2}), \nonumber \\
		W_6(x) &=& \frac{512 (k-2)^2 (k-1)^7 k^4 (3 k-2) x^{12 k-20}}{\left(x^{2 k}+x^2\right)^5} \mathcal{P}(x^{2k-2}), \nonumber
	\end{eqnarray}
where $P_{11,k}(x) =  ( 3 k (k (6 k-5)+1))x^{2}-(k+1)^2 x+3 (k-3) (2 k-3)$ and
		$$P_{12,k}(x) = A_{12,k} x^{3}+ B_{12,k}x^{2}+C_{12,k}x+D_{12,k}$$
with
	\[
	\begin{array}{l}
	A_{12,k}=-(k-5) (k-2) k (2 k-1) (3 k-4) (3 k-1), \\ 
	B_{12,k}=k (k (k (k (9 k (4 k+3)-562)+949)-498)+80)-8, \\ 
	C_{12,k}=(3 k-2) (k (k (2 k (18 k-89)+185)+10)-29),  \\ 
	D_{12,k}=(k-3) (2 k-3) (2 k-1) (3 k-2) (5 k-1).
	\end{array}
	\]
Straightforward computations show that $\Dis(P_{11,k})<0$ and
 $$\sgn(A_{12,k})= \sgn(B_{12,k})=\sgn(C_{12,k})=\sgn(D_{12,k})=1.$$ Therefore $W_i(x)\neq0$ in $(0, \infty)$ for $i\in\{0,1,2,3,4,5\}$. Applying Lemma \ref{prop.pol}, we get  that $W_6(x)$ has no zeros for $ k=\frac{4}{5}$  and one positive simple zero  for $k \in \left(\frac{3}{4},\frac{4}{5} \right) $.

\subsection*{Case 3:} Let  $k\in \left(\frac{1}{2},\frac{2}{3}\right) \cup \Big(\frac{2}{3},\frac{3}{4}\Big) \cup \left(\frac{4}{3},\frac{3}{2} \right)$. Thus, the Wronskians of $\F^k_{2}$ are given by
	\begin{eqnarray}
	W_0(x) &=& x^{3 k-2}, \nonumber \\
	W_1(x) &=&-2 k x^{4 k-5}, \nonumber \\
	W_2(x) &=& 8 (k-1) k x^{5 k-7}, \nonumber \\
	W_3(x) &=& 8 (k-1)^3 k x^{5 k-9} P_{13,k}(x^{2k-2}),\nonumber \\
	W_4(x) &=& \frac{256 (k-1)^6 k x^{10 k-11}}{\left(x^{2 k}+x^2\right)^3}P_{14,k}(x^{2k-2}), \nonumber \\
	W_5(x) &=& \frac{256 (k-2) (k-1)^6 k^2 (3 k-2) x^{10k-14}}{\left(x^{2 k}+x^2\right)^4}P_{15,k}(x^{2k-2}), \nonumber \\
	W_6(x) &=& \frac{512 (k-2)^2 (k-1)^7 k^4 (3 k-2) x^{12 k-20}}{\left(x^{2 k}+x^2\right)^5} \mathcal{P}(x^{2k-2}), \nonumber
	\end{eqnarray}
	with
	\begin{eqnarray}
	P_{13,k}(x) &=&  A_{13,k}x^{2}+B_{13,k} x+C_{13,k}, \nonumber \\
	P_{14,k} (x) &=& A_{14,k} x^{2}+B_{14,k} x+C_{14,k}, \nonumber \\
	P_{15,k}(x) &=& A_{15,k} x^{3}+ B_{15,k}x^{2}+C_{15,k}x+D_{15,k}, \nonumber
	\end{eqnarray}
	and
	\[
	\begin{array}{l}
	A_{13,k}= 3 k (3 k-1) , \\
	B_{13,k}=-(k+1)^2, \\ 
	C_{13,k}=-3 (k-3) , \\ 
	A_{14,k}=(k-5) (k-2) k (3 k-1),  \\  
	B_{14,k}=-2 (k (k (k (k+10)-30)+10)+1) , \\ 
	C_{14,k}=-(k-3) (2 k-1) (5 k-1), \\  
	A_{15,k}=(k-5) (k-2)^2 k (3 k-1) (4 k-3), \\ 
	B_{15,k}=-(k (k (k (k (k (82 k-329)+246)+341)-404)+84)+4), \\ 
	C_{15,k}=\left(k \left(5 k \left(k \left(4 k \left(k^2+4\right)-91\right)+84\right)-73\right)-16\right),  \\ 
	D_{15,k}=(k-3) (2 k-1) (5 k-4) (5 k-1).
	\end{array}
	\]
It is easy to check that $\Dis(P_{13,k})=k^4+112 k^3-354 k^2+112 k+1<0$,
 $\sgn(A_{14,k})=\sgn(B_{14,k})=\sgn(C_{14,k})\neq 0$ and
$$\sgn(A_{15,k})=\sgn(B_{15,k})=\sgn(C_{15,k})=\sgn(D_{15,k})\neq 0.$$
 This implies that $W_i(x)\neq0$ in $(0, \infty)$ for $i\in\{0,1,2,3,4,5\}$. Applying Lemma \ref{prop.pol},  we have  that $W_6(x)$ has no zeros for  $k \in \left(\frac{2}{3},\frac{3}{4} \right)$ and one positive simple zero for $k \in \left(\frac{1}{2}, \frac{2}{3} \right) \cup \left(\frac{4}{3},\frac{3}{2} \right)$.

\subsection*{Case 4:} Let $k\in \left(\frac{1}{5},k_2 \right) \cup \left(k_2,\frac{1}{3}\right) \cup \left(\frac{4}{5},1 \right) \cup \left(1,\frac{4}{3} \right] \cup \left(\frac{3}{2}, k_3 \right) \cup (k_3,2)$. Hence, the Wronskians of $\F^k_{2}$ are given by
	\begin{eqnarray}
	W_0(x) &=& x^{3 k-2}, \nonumber \\
	W_1(x) &=&-2 (k-1) x^{4 k-3}, \nonumber \\
	W_2(x) &=& -2 (k-1)^3 x^{4 k-4} P_{16,k}(x^{2k-2}), \nonumber \\
	W_3(x) &=& \frac{32 (k-1)^6 x^{9 k-7}}{\left(x^{2 k}+x^2\right)^2}P_{17,k}(x^{2k-2}),\nonumber \\
	W_4(x) &=& \frac{32 (k-2) (k-1)^6 k x^{11 k-11} }{\left(x^{2 k}+x^2\right)^3}P_{18,k}(x^{2k-2}), \nonumber \\
	W_5(x) &=& \frac{64 (k-2) (k-1)^7 k^2 (3 k-2) x^{11 k-14}}{\left(x^{2 k}+x^2\right)^4}P_{19,k}(x^{2k-2}), \nonumber \\
	W_6(x) &=& \frac{512 (k-2)^2 (k-1)^7 k^4 (3 k-2) x^{12 k-20}}{\left(x^{2 k}+x^2\right)^5} \mathcal{P}(x^{2k-2}), \nonumber
	\end{eqnarray}
	with
	\begin{eqnarray}
		P_{16,k}(x) &=&  A_{16,k}x^{2}+B_{16,k} x+C_{16,k}, \nonumber \\
		P_{17,k} (x) &=&  A_{17,k} x+B_{17,k}, \nonumber \\
		P_{18,k} (x) &=& A_{18,k} x^{2}+B_{18,k} x+C_{18,k}, \nonumber \\
		P_{19,k}(x) &=& A_{19,k} x^{3}+ B_{19,k}x^{2}+C_{19,k}x+D_{19,k}, \nonumber
	\end{eqnarray}
	and
	\[
	\begin{array}{l}
	A_{16,k}= 3 k, \\
 B_{16,k}=-(k+1), \\ 
	C_{16,k}=3 , \\ 
	 A_{17,k}=(k-5) k, \\ 
	B_{17,k}= 1-5 k, \\
	A_{18,k}=(k-5) k (2 k-1) (3 k-4),  \\  
	 B_{18,k}=-(k (k (51 k-98)+35)+4), \\ 
	  C_{18,k}= k ((71-30 k) k-43)+6, \\  A_{19,k}=(k-5) (k-2) k (2 k-1) (3 k-4) (4 k-3), \\ B_{19,k}=-(2 k-1) (3 k-4) (k (k (10 k (2 k-3)-49)+45)+2), \\ C_{19,k}=(5 k-4) (k (k (k (6 k+59)-125)+40)+8),  \\ D_{19,k}=(2 k-3) (3 k-2) (5 k-4) (5 k-1).
	\end{array}
	\]
Notice that $\Dis(P_{16,k})=k^2-34 k+1<0$, $\sgn(A_{17,k})=\sgn(B_{17,k})=-1$,
 \begin{equation}
\sgn(A_{18,k},B_{18,k},C_{18,k})=\left\{
\begin{array}{ll}
(1, 1,1),& k \in \left(\frac{4}{5}, 1\right) \cup \left(1,\frac{4}{3} \right), \nonumber \\
(-1,-1,-1),& k \in \left(\frac{1}{5},k_2 \right) \cup \left(k_2,\frac{1}{3} \right) \cup \left(\frac{3}{2}, k_3 \right) \cup (k_3,2), \nonumber \\
(1, 1, 0),& k=\frac{4}{3}, \nonumber 
\end{array}\right.
\end{equation}
 and
\begin{equation}
\sgn(A_{19,k},B_{19,k},C_{19,k}, D_{19,k})=\left\{
\begin{array}{ll}
(1, 1,1,1),& k \in \left( \left(\frac{1}{5},\frac{1}{3}\right) \cup \left(\frac{4}{5},\frac{4}{3} \right)\right)\setminus\{k_2, 1\} , \nonumber \\
(-1,-1,-1,-1),& k \in \left(\frac{3}{2}, k_3 \right) \cup (k_3,2), \nonumber \\
(1, 1, 0,0),& k=\frac{4}{3}. \nonumber 
\end{array}\right.
\end{equation}
 Thus, $W_i(x)\neq0$ in $(0, \infty)$ for $i\in\{0,1,2,3,4,5\}$. From  Lemma \ref{prop.pol},  we get  that $W_6(x)$ has no zeros for  $k \in \left(\frac{1}{5},k_2 \right) \cup \left(\frac{4}{5},1 \right) \cup \left(1,\frac{4}{3} \right] \cup \left(\frac{3}{2},k_3 \right)$ and two positive simple zeros for $k \in \left(k_2,\frac{1}{3} \right) \cup (k_3,2)$.

\subsection*{Case 5: } Let $k\in \left(k_1,\frac{1}{5} \right) \cup (2,3)$. Then, the Wronskians of $\F^k_{2}$ are given by
	\begin{eqnarray}
	W_0(x) &=& x^{2 k-2}, \nonumber \\
	W_1(x) &=&-k x^{3 k-5}, \nonumber \\
	W_2(x) &=& -2 (k-2) (k-1) k x^{3 k-7}, \nonumber \\
	W_3(x) &=& 4 (k-2)^2 (k-1) k^2 x^{4 k-10},\nonumber \\
	W_4(x) &=& 4 (k-2)^2 (k-1)^2 k^2 x^{4 k-13} P_{20,k}(x^{2k-2}), \nonumber \\
	W_5(x) &=& \frac{16 (k-2)^2 (k-1)^4 k^2 x^{7 k-12}}{\left(x^{2 k}+x^2\right)^4}P_{21,k}(x^{2k-2}), \nonumber \\
	W_6(x) &=& \frac{512 (k-2)^2 (k-1)^7 k^4 (3 k-2) x^{12 k-20}}{\left(x^{2 k}+x^2\right)^5} \mathcal{P}(x^{2k-2}), \nonumber
	\end{eqnarray}
	with
	\begin{eqnarray}
	P_{20,k}(x) &=&  A_{20,k}x^{2}+B_{20,k} x+C_{20,k}, \nonumber \\
	P_{21,k}(x) &=& A_{21,k} x^{6}+ B_{21,k}x^{5}+C_{21,k}x^4+D_{21,k} x^3+E_{21,k} x^{2}+ F_{21,k}x+G_{21,k}, \nonumber
	\end{eqnarray}
	and
	\[
	\begin{array}{l} 
	A_{20,k}= 3 k (2 k-1) (3 k-1) (4 k-3),\\
	B_{20,k}= (k+1)^2 (2 k-1),  \\
	C_{20,k}= (9-2 k) k-9,\\
	A_{21,k}=	3 k^4 (2 k-1) (3 k-2) (3 k-1) (4 k-3), \\ 
	B_{21,k}=k^3 (2 k-1) (3 k-2) (k (k (144 k-157)+34)-1), \\ 
	C_{21,k}=k (k (k (k (k (6 k (12 k (16 k-17)-877)+15977)-19731)+12594)-3976)+480),  \\ 
	D_{21,k}=2 (k (k (k (k (k (4 k (9 k (28 k-123)+1928)-4375)-3482)+5733)-2528)+324)+16), \\
	E_{21,k}=-(3 k-2) (k (k (k (k (k (2 k (36 k-77)+337)-1899)+3376)-2038)+312)+64),\\
	F_{21,k}=(3 k-2) (k (k (k (k (k (2 k-153)+894)-1625)+1234)-444)+48) ,\\
	G_{21,k}=-(k-3) k^2 (2 k-3) (3 k-2) (3 k+2).
	\end{array}
	\]

Clearly $W_i(x)\neq0$ in $(0,\infty)$ for $i \in\{0,1,2,3\}$. Straightforward computations show that $\sgn(A_{20,k})=\sgn(B_{20,k})=\sgn(C_{20,k})\neq0$,
\begin{equation}
\sgn(A_{21,k},B_{21,k},C_{21,k}, D_{21,k},  E_{21,k}, F_{21,k}, G_{21,k})=\left\{
\begin{array}{ll}
(1, 1,1,1,1,1),& k \in \left(k_1, \frac{1}{5} \right), \nonumber \\
(1,1,1,1, \sgn(E_{21,k}),1,1) ,& k \in (2, 3), \nonumber \\
\end{array}\right.
\end{equation}
 $\sgn(\Dis(P_{21,k}))=-\sgn(\Dis(P''_{21,k}))=1$ and $P_{21,k}(-1)<0$ for $k \in (2,3)$. This implies that $W_4, W_5\neq0$ in $(0, \infty)$. From  Lemma \ref{prop.pol},  we get  that $W_6$ has one and two positive zeros for   $k \in \left(k_1, \frac{1}{5} \right)$ and  $k \in (2,3)$, respectively.

	\subsection*{Case 6: } Let $k\in \left(\frac{1}{3},\frac{1}{2} \right) \cup (4,5]$. The Wronskians of $\F^k_{2}$ are given by \eqref{l10}. Observe that
	\begin{equation}
		\sgn(A_{9,k},B_{9,k},C_{9,k})=\left\{
		\begin{array}{ll}
			(1, 1,1),& k \in \left(\frac{1}{3}, \frac{1}{2} \right), \nonumber \\
			(1, -1,1) ,& k \in (4, 5], \nonumber \\
		\end{array}\right.
	\end{equation}
	and $\Dis(P_{9,k})=-(2 k-1) (862 k^5-4831 k^4+8308 k^3-5186 k^2+974 k+1)<0$ if $k>4$. So, $W_i\neq0$ in $(0, \infty)$ for $i\in\{0,1,2,3,4,5\}$. From  Lemma \ref{prop.pol},  we get  that $W_6(x)$ has one positive simple zero  for $k \in \left(\frac{1}{3},\frac{1}{2}\right) \cup (4,5]$.
\end{proof}

\begin{proposition} \label{prop9}
 For $k \in (5, \infty)$, $\z(\G_9^k)=7$. 
\end{proposition} 

The proof of Proposition \ref{prop9} is analogous to the proof of Proposition 7  in \cite{AndCesCruNov2021} and will be omitted.

\section{Proofs}
In order to prove Theorems \ref{Theorem-Melnikov2}, \ref{Theorem-Melnikov21} and \ref{Theorem-Melnikov22}, we need the following lemma.

\begin{lemma}\label{lem0}
If $i,m,n>0$ then $m_{i}(m, n)=m_{i}(n,m)$.
\end{lemma}
\begin{proof}
The proof is immediate by using the   linear change of coordinates $(x,y)\rightarrow (y,x).$
\end{proof}

\subsection*{Proof of Theorem \ref{Theorem-Melnikov2}}
Using the polar coordinates $x= r\cos(\theta)$ and $y=r\sin(\theta)$, system \eqref{general-system1} becomes equivalent to
\begin{equation}\label{polar-system}
(\dot{r}, \dot{\theta})^T=(0,-1)^T+\displaystyle\sum_{i=1}^3 \varepsilon^i M_i(\theta, r),
\end{equation}
with
\[
M_i(r)=\left\lbrace 
\begin{array}{lll}
(A_i^+(r, \theta), B_i^+(r, \theta))^T,& \text{if}& \sin (\theta)-r^{n-1} \cos^n(\theta)>0,\\
(A_i^-(r, \theta), B_i^-(r, \theta))^T, & \text{if}& \sin (\theta)-r^{n-1} \cos^n(\theta)<0,\\
\end{array}\right. 
\]
where
\[
\begin{array}{l}
A_i^+=\cos (\theta ) (a_{0i}+r (a_{2i}+b_{1i}) \sin (\theta ))+a_{1i} r \cos ^2(\theta )+\sin (\theta ) (b_{0i}+b_{2i} r \sin (\theta )),\\
B_i^+= r^{-1}[-\sin (\theta ) (a_{0i}+a_{2i} r \sin (\theta ))+\cos (\theta ) (r (b_{2i}-a_{1i}) \sin (\theta )+b_{0i})+b_{1i} r \cos ^2(\theta )],\\
A_i^-=\cos (\theta ) (\alpha_{0i}+r (\alpha_{2i}+\beta_{1i}) \sin (\theta ))+\alpha_{1i} r \cos ^2(\theta )+\sin (\theta ) (\beta_{0i}+\beta_{2i} r \sin (\theta )),\\
B_i^-= r^{-1}[-\sin (\theta ) (\alpha_{0i}+\alpha_{2i} r \sin (\theta ))+\cos (\theta ) (r (\beta_{2i}-\alpha_{1i}) \sin (\theta )+\beta_{0i})+\beta_{1i} r \cos ^2(\theta )].
\end{array}
\]
Taking $\theta$ as the new indepent variable, system \eqref{polar-system} writes as
\begin{equation} \label{polar-systemB}
\dfrac{\text{d} r}{\text{d}\theta}=\left\lbrace 
\begin{array}{lll}
\displaystyle\sum_{i=1}^3\varepsilon^i F_i^+(r, \theta)+ \mathcal{O}(\varepsilon^4), & \text{if}& \sin (\theta)-r^{n-1} \cos^n(\theta)>0,\\
\displaystyle\sum_{i=1}^3\varepsilon^i F_i^-(r, \theta)+ \mathcal{O}(\varepsilon^4), & \text{if}& \sin (\theta)-r^{n-1} \cos^n(\theta)<0,\\
\end{array}\right. 
\end{equation}
where
\[
\begin{array}{l}
F_1^+(r,\theta) = - \cos (\theta ) (a_{01}+r (a_{21}+b_{11}) \sin (\theta ))-a_{11} r \cos ^2(\theta )-\sin (\theta ) (b_{01}+b_{21} r \sin (\theta )),\\
F_1^-(r,\theta) = -\cos (\theta ) (\alpha_{01}+r (\alpha_{21}+\beta_{11}) \sin (\theta ))-\alpha_{11} r \cos ^2(\theta )-\sin (\theta ) (\beta_{01}+\beta_{21} r \sin (\theta )).
\end{array}
\]

 Let $n, m$ be odd positive integers. So, system \eqref{polar-systemB} becomes equivalent to
\begin{equation} \label{polar-systemB1}
\dfrac{\text{d} r}{\text{d}\theta}=\left\lbrace 
\begin{array}{lr}
\displaystyle\sum_{i=1}^3\varepsilon^i F_i^-(r, \theta)+ \mathcal{O}(\varepsilon^4), &0<\theta(r)< \theta_1(r),\\
\displaystyle\sum_{i=1}^3\varepsilon^i F_i^+(r, \theta)+ \mathcal{O}(\varepsilon^4), &\theta_1(r)<\theta(r)< \pi+\theta_1(r),\\
\displaystyle\sum_{i=1}^3\varepsilon^i F_i^-(r, \theta)+ \mathcal{O}(\varepsilon^4), & \pi+\theta_1(r)<\theta(r)< 2\pi,\\
\end{array}\right. 
\end{equation}
  where $\theta_1(r)$ is the solution of the equation $r^n \sin^n (\theta)-r^m \cos ^{m} \theta=0$ in $\left[0, \frac{\pi}{2} \right]$. Through straightforward computations and  taking $x=r \cos(\theta_1(r)),$ we obtain that the first order Melnikov function of system \eqref{polar-systemB1} is given by
$$\Delta_1(x)=\frac{\varrho_1(x)}{2x^{-1} \sqrt{x^{2k}+x^2}}, \qquad x>0,$$ \\\
where $k=\dfrac{m}{n} \in \Q^+ \setminus  \left\{\dfrac{p}{q} \in \Q: (p,q)=1, (pq,2)=2\right\} $,
\begin{equation*}
\varrho_1(x)=
\left\{\begin{array}{ll}
v_0 u_2^k(x)+v_1 u_{8}^k(x) +v_2 u_0^k(x), & k\neq 1,\\
(v_0 +v_2)u_0^k(x)+v_1 u_{8}^k(x) , & k=1,\\
\end{array}\right.
\end{equation*} 
 and 
\begin{equation*}
\begin{array}{ll}
v_0=& 4 (a_{01}-\alpha_{01});\\
v_1=&-\pi(a_{11}+\alpha_{11}+b_{21}+\beta_{21});\\
v_2=&-4( b_{01}- \beta_{01}).
\end{array}
\end{equation*}
From Theorem \ref{thm:melnikov}, $m_1(m,n)$ coincides with the maximum number of positive zeros of the function $\varrho_1(x)$.
Notice that  $\varrho_1(x) \in \Span( \G_1^k)$  for $k=1$  and  $\varrho_1(x) \in \Span( \G_3^k)$  for $k\neq 1$. Moreover, $\{v_0, v_1,v_2\}$ and $\{v_0+v_2, v_1\}$ are linearly independent. Hence from Lemma \ref{lem0}, by statements  \textit{(2), (4)} and \textit{(5)} of Proposition \ref{prop5} and Theorem \ref{thm:melnikov}, we obtain that
\begin{equation*}
m_1(m,n)=
\left\{\begin{array}{ll}
1, & \dfrac{m}{n}=1, \\
\\
2, & \dfrac{m}{n} \in  \left[\dfrac{1}{2},1\right) \cup  (1,2],\\
\\
3, & \dfrac{m}{n} \in \left(0,\dfrac{1}{2}\right) \cup (2,\infty).
\end{array}\right.
\end{equation*}
Observe that  $\Delta_1$ is identically zero if and only if $(a_{01},a_{11},b_{01})=( \alpha_{01}, -\alpha_{11}-b_{21}-\beta_{21}, \beta_{01})$ for $k\neq1$ and $(a_{01}, a_{11})=(\alpha_{01}+b_{01}-\beta_{01},-\alpha_{11}-b_{21}-\beta_{21})$ for $k=1$. In this case, the second order Melnikov   function of system \eqref{polar-systemB1} is given by
$$\Delta_2(x)=\frac{\varrho_2(x)}{4x^{-3} \sqrt{x^{2k}+x^2} \left(m x^{2k}+n x^2\right)}, \qquad x>0,$$ \\\
with 
\begin{equation*}
\varrho_2(x)=
\left\{\begin{array}{ll}
\ell_{0} u_{9}^k(x)+\ell_{1} u_2^k(x)+\ell_{2} u_5^k(x)+\ell_{3} u_{4}^k(x)+\ell_{4} u_0^k(x)+\ell_{5} u_1^k(x)+\ell_{6} u_6^k(x) + & k\neq 1,\\
+\ell_{7} u_7^k(x), &\\
\ell_8 u_{0}^k(x)+\ell_9 u_{1}^k(x) , & k= 1,\\
\end{array}\right.
\end{equation*} 
 and 
\begin{equation*}
\begin{array}{ll}
\ell_0=&-\pi n (-\alpha_{11} \alpha_{21}+\alpha_{11} \beta_{11}+2 a_{12}+2 \alpha_{12}+\alpha_{11} a_{21}+a_{21} \beta_{21}-\alpha_{21} \beta_{21}-\alpha_{11} b_{11}-b_{11} \beta_{21}+\\
&\beta_{11} \beta_{21}+2 b_{22}+2 \beta_{22});\\
\ell_1=&8 (a_{02} n+\alpha_{01} b_{11} n-\beta_{01} b_{21} m-\beta_{01} b_{21} n-2 \alpha_{11} \beta_{01} m-\beta_{01} \beta_{21} m-\alpha_{01} \beta_{11} n-\alpha_{02} n-\\
&2 \alpha_{11} \beta_{01} n-\beta_{01} \beta_{21} n);\\
\ell_2=&8 (a_{21} \beta_{01} m-b_{02} m+\alpha_{01} b_{21} (m+n)-\alpha_{01} \beta_{21} m-\alpha_{21} \beta_{01} m+\beta_{02} m-\alpha_{01} \beta_{21} n);\\
\ell_3=&-\pi  (a_{21} (\alpha_{11}+\beta_{21}) (3 m-n)+\alpha_{11} b_{11} m+b_{11} \beta_{21} m-3 \alpha_{11} b_{11} n-3 b_{11} \beta_{21} n+2 b_{22} (m+n)-\\
&3 \alpha_{11} \alpha_{21} m-\alpha_{11} \beta_{11} m+2 \alpha_{12} m-3 \alpha_{21} \beta_{21} m-\beta_{11} \beta_{21} m+2 \beta_{22} m+\alpha_{11} \alpha_{21} n+3 \alpha_{11} \beta _{11} n+\\
&2 \alpha_{12} n+\alpha_{21} \beta_{21} n+3 \beta_{11} \beta_{21} n+2 \beta_{22} n + 2 a_{12} (m+n)); \\
\ell_4= & -8 n (2 \alpha_{01} (\alpha_{11}+\beta_{21})-\beta_{01} \beta_{11}+b_{02}-\beta_{02}+\beta_{01} b_{11}); \\
\ell_5=&2 \pi  (\alpha_{11}+\beta_{21}) (m-n) (2 \alpha_{11}+b_{21}+\beta_{21});\\
\ell_6=&-2 \pi  (\alpha_{11}+\beta_{21}) (b_{21}-\beta_{21}) (m-n);\\
\ell_7=&8 m (\alpha_{01} \alpha_{21}+a_{02}-\alpha_{02}-2 \alpha_{11} \beta_{01}-\alpha_{01} a_{21}-2 \beta_{01} \beta_{21});\\
\ell_8=& 8 n (-2 \alpha_{01} \alpha_{11}+\alpha_{01} \alpha_{21}-\alpha_{01} \beta_{11}+2 a_{02}-2 \alpha_{02}-2 \alpha_{11} \beta_{01}-\alpha_{01} a_{21}+a_{21} \beta_{01}-\alpha_{21} \beta_{01} \\&-4 \beta_{21} (\alpha_{01}+b_{01})-4 \alpha_{11} b_{01}+\beta_{01} \beta_{11}-2 b_{02}+2 \beta_{02}+\alpha_{01} b_{11}-\beta_{01} b_{11}+2 \alpha_{01} b_{21}-2 \beta_{01} b_{21}) ;\\
\ell_9=& -4 \pi  n (2 (\alpha_{12}+\beta_{22})+2 a_{12}+(\alpha_{11}+\beta_{21}) (a_{21}-\alpha_{21}-b_{11}+\beta_{11})+2 b_{22}).
\end{array}
\end{equation*}
Notice that $\varrho_2(x) \in \Span( \G_1^k)$  for $k=1$  and  $\varrho_2(x) \in \Span( \G_5^k)$  for $k \neq1$. Observe also that the sets $\{\ell_0, \ell_1,\ell_2,\ell_3,\ell_4,\ell_5,\ell_6,\ell_7\}$ and $\{\ell_8,\ell_9\}$ are linearly independent. Hence from Lemma \ref{lem0}, by statements  \textit{(2)},\textit{(8)} and \textit{(9)} of Proposition \ref{prop5} and Theorem \ref{thm:melnikov}, we obtain that
\begin{equation*}
m_2(m,n)=
\left\{\begin{array}{ll}
1, & \dfrac{m}{n}=1, \\
\\
7, & \dfrac{m}{n} \in  \left(0,\dfrac{1}{2} \right) \cup \left(\dfrac{2}{3},\dfrac{3}{4}\right] \cup \left[\dfrac{4}{3},\dfrac{3}{2}\right) \cup (2, \infty),\\
\\
8, & \dfrac{m}{n} \in \left(\dfrac{1}{2},\dfrac{2}{3}\right) \cup \left(\dfrac{3}{4}, 1\right) \cup \left(1, \dfrac{4}{3}\right) \cup \left(\dfrac{3}{2},2\right).
\end{array}\right.
\end{equation*}
Imposing that $\Delta_1=\Delta_2=0$, we obtaind that the third order Melnikov function of system \eqref{polar-systemB1} 
can be written as
$$\Delta_3(x)=\frac{\varrho_3(x)}{\nu(x)}, \qquad x>0,$$ \\\
where $\nu(x)\neq 0$ in $(0, \infty)$ and
\begin{equation*}
\varrho_3(x)=
\left\{\begin{array}{ll}
\omega_{0} u_{9}^k(x)+\omega_{1} u_2^k(x)+\omega_{2} u_5^k(x)+\omega_{3} u_{4}^k(x)+\omega_{4} u_0^k(x)+\omega_{5} u_1^k(x)+  & k\neq 1,\\
\omega_{6} u_6^k(x)+\omega_{7} u_7^k(x), &\\
\omega_8 u_{0}^k(x)+\omega_9 u_{1}^k(x)+\omega_{10} u_{13}^k(x) , & k= 1.\\
\end{array}\right.
\end{equation*} 
The explicit expressions of $\omega_i$, for $i=0,1,2,3$, are very long, therefore we shall omit them here. However,  it is easy to check  that the sets $\{\omega_0, \omega_1,\omega_2,\omega_3,\omega_4,\omega_5,\omega_6,\omega_7\}$ and $\{\omega_8,\omega_9, \omega_{10}\}$ are linearly independent. Therefore, $m_3(n,m)=2$ for $\frac{m}{n}=1$ and  $m_2(n,m)=m_3(m,n)$ for $\frac{m}{n}\neq1.$

Hence, we have concluded the proof of Theorem \ref{Theorem-Melnikov2}.

\subsection*{Proof of Theorem \ref{Theorem-Melnikov21}}
 Let $m$ and $n$ be even positive integers.  Analogously to the previous proof, we get that system \eqref{polar-systemB} becomes equivalent to
\begin{equation} \label{polar-system3}
\dfrac{\text{d} r}{\text{d}\theta}=\left\lbrace 
\begin{array}{lr}
\displaystyle\sum_{i=1}^2\varepsilon^i F_i^-(r, \theta)+ \mathcal{O}(\varepsilon^3), &0<\theta(r)< \theta_1(r),\\
\displaystyle\sum_{i=1}^2\varepsilon^i F_i^+(r, \theta)+ \mathcal{O}(\varepsilon^3), & \theta_1(r)<\theta(r)<\pi-\theta_1(r),\\
\displaystyle\sum_{i=1}^2\varepsilon^i F_i^-(r, \theta)+ \mathcal{O}(\varepsilon^3), & \pi-\theta_1(r) <\theta(r)< \pi+\theta_1(r),\\
\displaystyle\sum_{i=1}^2\varepsilon^i F_i^+(r, \theta)+ \mathcal{O}(\varepsilon^3), & \pi+\theta_1(r)<\theta(r)< 2\pi-\theta_1(r),\\
\displaystyle\sum_{i=1}^2\varepsilon^i F_i^-(r, \theta)+ \mathcal{O}(\varepsilon^3), & 2\pi-\theta_1(r)<\theta(r)< 2\pi.\\
\end{array}\right. 
\end{equation}
Where $\theta_1(r)$  is the solution of the equation $r^n \sin^n (\theta)-r^m \cos ^{m} \theta=0$ in $\left[0, \frac{\pi}{2} \right]$. Taking $x=r \cos(\theta_1(r)),$ we get that the first order Melnikov function of system \eqref{polar-system3} is given by
$$\D_1(x)=\frac{\varrho_4(x) }{x^{-1}\sqrt{x^{2k}+x^2}}, \qquad x>0,$$ \\
where $k=\dfrac{m}{n} \in \Q^+$,
\begin{equation*}
\varrho_4(x)=
\left\{\begin{array}{ll}
 v_0 u_1^k(x) +v_1 u_{8}^k(x)+v_2 u_{11}^k(x), & k\neq 1,\\
 (v_0+2v_1+2v_2) u_{1}^k(x), & k= 1,\\
\end{array}\right.
\end{equation*} 
 and 
\begin{equation*}
\begin{array}{ll}
v_0=& 2(a_{11}-\alpha _{11}-b_{21}+\beta_{21});\\
v_1=&-\pi(a_{11}+b_{21});\\
v_2=&2(a_{11}-\alpha_{11}+b_{21}-\beta_{21}).
\end{array}
\end{equation*}
Observe that  $q_4(x) \in \Span( \G_0^k)$  for $k=1$  and  $q_4(x) \in \Span( \G_2^k)$  for $k \neq1$. Furthermore, $\{v_0, v_1,v_2\}$ is linearly independent. Hence from Lemma \ref{lem0}, by statements  \textit{(1)} and \textit{(3)} of Proposition \ref{prop5} and Theorem \ref{thm:melnikov}, we obtain that
\begin{equation*}
m_1(m,n)=
\left\{\begin{array}{ll}
0, & \dfrac{m}{n}= 1,\\
\\
2, & \dfrac{m}{n} \neq 1.
\end{array}\right.
\end{equation*}
 In this case,  $\Delta_1=0$  if and only if $(a_{11},b_{21}, \alpha_{11})=( -\beta_{21}, \beta_{21},- \beta_{21})$ for $k\neq1$ and $\alpha_{11}=(2+\pi)^{-1}( -\pi  a_{11}+2 a_{11}-2 b_{21}-\pi  \beta_{21}+2 \beta_{21}-\pi b_{21})$ for $k=1$. Therefore, the second order Melnikov   function of system \eqref{polar-system3} is given by
$$\Delta_2(x)=\frac{\varrho_5(x)}{x^{-3} \sqrt{x^{2k}+x^2} \left(m x^{2k}+n x^2\right)}, \qquad x>0,$$ \\\
with
\begin{equation*}
\varrho_5(x)=
\left\{\begin{array}{ll}
\ell_0 u_{10}^{k}(x)+ \ell_1 u_{1}^{k}(x) + \ell_2 u_{12}^{k}(x) + \ell_3 u_{6}^{k}(x)+ \ell_4 u_{3}^{k}(x), & k\neq 1,\\
\ell_5 u_3^k(x) + \ell_6 u_1^k(x) , & k= 1,\\
\end{array}\right.
\end{equation*} 
 and 
\begin{equation*}
\begin{array}{ll}
\ell_0=&-\pi (a_{12}+b_{22});\\
\ell_1=&2 n (a_{12}-\alpha_{12}-2 b_{11} \beta_{21}+2 \beta_{11} \beta_{21}-b_{22}+\beta_{22});\\
\ell_2=&2n (a_{12}-\alpha_{12}+b_{22}-\beta_{22});\\
\ell_3=&2 m (a_{12}-\alpha_{12}+2 a_{21} \beta_{21}-2 \alpha_{21} \beta_{21}-b_{22}+\beta_{22});\\
\ell_4=&4 a_{01} b_{01} m-4 \alpha_{01} (b_{01} (m-n)+\beta_{01} n); \\
\ell_5=& 4 n (a_{01} b_{01}-\alpha_{01} \beta_{01}); \\
\ell_6=& -n (2+\pi)^{-1} (a_{11} ((\pi -2) (\pi \alpha_{21}-(4+\pi ) \beta_{11})-\pi  (2+\pi ) a_{21}+(\pi -4) (2+\pi ) b_{11}) \\& +2 \left(\pi ^2-4\right) a_{12}+2 \pi ^2 \alpha_{12}+8 \pi  \alpha_{12}+8 \alpha_{12}-\pi ^2 a_{21} b_{21}-8 a_{21} b_{21}-6 \pi a_{21} b_{21} \\&+8 \alpha_{21} \beta_{21}+\pi ^2 b_{11} b_{21}+2 \pi b_{11} b_{21}+8 \beta_{11} \beta_{21}+\pi ^2 \alpha_{21} b_{21}+2 \pi  \alpha_{21} b_{21}-\pi \beta_{11} b_{21}\\&-6 \pi  \beta_{11} b_{21}-8 \beta_{11} b_{21}+2 \pi ^2 b_{22}+8 b_{22}+2 \left(\pi ^2-4\right) \beta_{22}+8 \pi  b_{22}).
\end{array}
\end{equation*}
Thus, $\varrho_5(x) \in \Span( \G_{11}^k)$  for $k=1$,  $q_5(x) \in \Span( \G_8^k)$ for $k \neq1$. Moreover, the sets $\{\ell_0, \ell_1,\ell_2,\ell_3,\ell_4\}$ and $\{\ell_5,\ell_6\}$ are linearly independent. Hence from Lemma \ref{lem0}, by item  \textit{(13)} of Proposition \ref{prop5}, by the statements  \textit{(1)} and \textit{(2)} of Proposition \ref{prop6} and Theorem \ref{thm:melnikov}, we obtain that
\begin{equation*}
m_2(m,n)=
\left\{\begin{array}{ll}
1, & \dfrac{m}{n}=1, \\
\\
4, & \dfrac{m}{n} \in  \left[\dfrac{1}{5},\dfrac{1}{3} \right] \cup \left[\dfrac{1}{2},1\right) \cup (1,2] \cup [3,5],\\
\\
5, & \dfrac{m}{n} \in \left(0,\dfrac{1}{5}\right) \cup \left(\dfrac{1}{3}, \dfrac{1}{2}\right) \cup \left(2, 3\right) \cup \left(5,\infty\right).
\end{array}\right.
\end{equation*}

Hence, we have concluded the proof of Theorem \ref{Theorem-Melnikov21}.

\subsection*{Proof of Theorem \ref{Theorem-Melnikov22}}
 Let $ m$ be even and $n$ be odd positive integers. Applyng the polar coordinates $x= r\cos(\theta)$ and $y=r\sin(\theta)$, system \eqref{polar-systemB} becomes equivalent to
\begin{equation} \label{polar-system4}
\dfrac{\text{d} r}{\text{d}\theta}=\left\lbrace 
\begin{array}{lr}
\displaystyle\sum_{i=1}^3\varepsilon^i F_i^-(r, \theta)+ \mathcal{O}(\varepsilon^4), &0<\theta(r)< \theta_1(r),\\
\displaystyle\sum_{i=1}^3\varepsilon^i F_i^+(r, \theta)+ \mathcal{O}(\varepsilon^4), & 0<\theta(r)<2\pi- \theta_1(r),\\
\displaystyle\sum_{i=1}^3\varepsilon^i F_i^-(r, \theta)+ \mathcal{O}(\varepsilon^4), & 2\pi- \theta_1(r)<\theta(r)< 2\pi.\\
\end{array}\right. 
\end{equation}
The first order Melnikov function of system \eqref{polar-system4} is given by
$$\D_1(x)=\frac{\varrho_6(x)}{x^{-1} \sqrt{x^{2k}+x^2}}, \qquad x>0,$$ \\
where $k=\dfrac{m}{n} \in \Q^+ \setminus  \left\{\dfrac{p}{q} \in \Q: (p,q)=1, (p,2)=1\right\} $,
\begin{equation*}
\varrho_6(x)=
\begin{array}{ll}
v_0 u_2^k(x)+v_1 u_1^k(x) +  v_2 u_{8}^k(x)+v_3 u_{11}^k(x)\\
\end{array}
\end{equation*} 
 and 
\begin{equation*}
\begin{array}{ll}
v_0=& 2 (a_{01}-\alpha_{01});\\
v_1=&a_{11}-\alpha_{11}-b_{21}+\beta_{21};\\
v_2=&-\pi (a_{11}+b_{21});\\
v_3=& a_{11}-\alpha_{11}+b_{21}-\beta_{21}.
\end{array}
\end{equation*}
Thus, $\varrho_6(x) \in \Span( \G_4^k)$. Notice that $\{v_0, v_1,v_2,v_3\}$ is linearly independent. Hence by statements  \textit{(6)} and \textit{(7)} of Proposition \ref{prop5} and Theorem \ref{thm:melnikov}, we obtain that
\begin{equation*}
m_1(m,n)=
\left\{\begin{array}{ll}
3, & \dfrac{m}{n} \in  \left[\dfrac{1}{2},1\right) \cup  \left(1,\dfrac{3}{2}\right] \cup [2, \infty),\\
\\
4, & \dfrac{m}{n} \in \left(0,\dfrac{1}{2}\right) \cup \left(\dfrac{3}{2},2\right).
\end{array}\right.
\end{equation*}

Notice that  $\Delta_1=0$  if and only if $(b_{01},a_{11},b_{21}, \alpha_{11})=(\beta_{01}, -\beta_{21}, \beta_{21},- \beta_{21})$. Then the second order Melnikov  function of system \eqref{polar-system4} is given by
$$\Delta_2(x)=\frac{\varrho_7(x)}{2 x^{-3} \sqrt{x^{2k}+x^2} \left(m x^{2k}+n x^2\right)}, \qquad x>0,$$ \\\
with
\begin{equation*}
\varrho_7(x)= \left\{\begin{array}{ll}
\omega_0u_{10}^k(x)+\omega_1u_{5}^k(x)+\omega_2u_{12}^k(x)+\omega_3u_{6}^k(x)+\omega_4u_{3}^k(x)+\omega_5u_{0}^k(x)+\omega_6u_{1}^k(x), & k \neq \frac{2}{3},2, \\
\ell_0u_{10}^k(x)+\ell_1u_{5}^k(x)+\ell_2u_{12}^k(x)+(\ell_3+\ell_5)u_{0}^k(x)+\ell_4u_{3}^k(x)+\ell_6u_{1}^k(x), & k=\frac{2}{3}, \\
\ell_0u_{10}^k(x)+(\ell_1+\ell_6)u_{5}^k(x)+\ell_2u_{12}^k(x)+\ell_3u_{6}^k(x)+(\ell_4+\ell_5)u_{0}^k(x), & k=2, \\
\end{array} \right.
\end{equation*} 
 and 
\begin{equation*}
\begin{array}{ll}
\ell_0=&-\pi n (a_{12}+\alpha_{12}+b_{22}+\beta_{22});\\
\ell_1= &4 m (a_{01} \beta_{21}-\alpha_{01} \beta_{21}+a_{21} \beta_{01}-\alpha_{21} \beta_{01}-b_{02}+\beta_{02});\\
\ell_2=& 2n (a_{12}-\alpha_{12}+b_{22}-\beta_{22});\\
\ell_3=& 2 m (a_{12}-\alpha_{12}+2 a_{21} \beta_{21}-2 \alpha_{21} \beta_{21}-b_{22}+\beta_{22});\\
\ell_4= &4 m \beta_{01}(a_{01}-\alpha_{01}); \\
\ell_5= &4 n (\beta_{01} \beta_{11}-b_{02}+\beta_{02}-\beta_{01} b_{11}); \\
\ell_6= &2 n (a_{12}-\alpha_{12}-2 b_{11} \beta_{21}+2 \beta_{11} \beta_{21}-b_{22}+\beta_{22}).
\end{array}
\end{equation*}

Thus, $\varrho_7(x) \in \Span( \G_9^k)$. Moreover, the sets $\{\ell_0, \ell_1,\ell_2,\ell_3,\ell_4,\ell_5,\ell_6\}$, $\{\ell_0, \ell_1,\ell_2,\ell_3+\ell_5,\ell_4,\ell_6\}$ and $\{\ell_0, \ell_1+\ell_6,\ell_2,\ell_3,\ell_4+\ell_5\}$  are linearly independent. Hence by statements \textit{(10)} and \textit{(11)} of Proposition \ref{prop5}, by propositions \ref{prop8} and \ref{prop9} and Theorem \ref{thm:melnikov}, we obtain that
\begin{equation*}
m_2(m,n)=
\left\{\begin{array}{ll}
4, & \dfrac{m}{n}=2, \\
\\
5, & \dfrac{m}{n}=\dfrac{2}{3}, \\
\\
6, & \dfrac{m}{n} \in \left(\dfrac{1}{5}, k_2 \right) \cup \left(\dfrac{2}{3},\dfrac{3}{4} \right) \cup \left[\dfrac{4}{5},1 \right) \cup \left(1,\dfrac{4}{3} \right] \cup \left(\dfrac{3}{2}, k_3 \right),\\
\\
7, & \dfrac{m}{n} \in \left(k_1,\dfrac{1}{5} \right) \cup \left(\dfrac{1}{3},\dfrac{1}{2} \right) \cup \left(\dfrac{1}{2},\dfrac{2}{3} \right) \cup \left(\dfrac{3}{4},\dfrac{4}{5} \right) \cup \left(\dfrac{4}{3},\dfrac{3}{2} \right) \cup (k_5, \infty),
\end{array}\right.
\end{equation*}
and
\begin{equation*}
m_2(m,n) \in 
\left\{\begin{array}{ll}
7 \leq m_2(m,n) \leq 8, & \dfrac{m}{n} \in  \left(k_2,\dfrac{1}{3} \right) \cup (k_3,2) \cup (2,3), \\
\\
7 \leq m_2(m,n) \leq 9, & \dfrac{m}{n} \in (k_4,k_5), \\
\\
6 \leq m_2(m,n) \leq 11, & \dfrac{m}{n} \in (0,k_0) \cup (k_0,k_1) \cup (3,k_4).
\end{array}\right.
\end{equation*}
 Computing the third order Melnikov  function of system \eqref{polar-system4} we obtain that
$$\Delta_3(x)=\frac{\varrho_9(x)}{\nu_2(x)}, \qquad x>0,$$ \\\
with $\nu(x)\neq 0$ in $(0, \infty)$ and
\begin{equation*}
\varrho_9(x)= \left\{\begin{array}{ll}
\ell_0u_{10}^k(x)+\ell_1u_{5}^k(x)+\ell_2u_{12}^k(x)+\ell_3u_{6}^k(x)+\ell_4u_{3}^k(x)+\ell_5u_{0}^k(x)+\ell_6u_{1}^k(x), & k \neq \frac{2}{3},2, \\
\ell_0u_{10}^k(x)+\ell_1u_{5}^k(x)+\ell_2u_{12}^k(x)+(\ell_3+\ell_5)u_{0}^k(x)+\ell_4u_{3}^k(x)+\ell_6u_{1}^k(x), & k=\frac{2}{3}, \\
\ell_0u_{10}^k(x)+(\ell_1+\ell_6)u_{5}^k(x)+\ell_2u_{12}^k(x)+\ell_3u_{6}^k(x)+(\ell_4+\ell_5)u_{0}^k(x), & k=2. \\
\end{array} \right.
\end{equation*} 
The explicit expressions of $\omega_i$, for $i=0,1,2,3$, are very long, therefore we shall omit them here. However,  it is easy to check  that the sets $\{\omega_0, \omega_1,\omega_2,\omega_3,\omega_4,\omega_5,\omega_6,\omega_7\}$ and $\{\omega_8,\omega_9, \omega_{10}\}$ are linearly independent. Therefore, $m_3(n,m)=2$ for $\frac{m}{n}=1$ and  $m_2(n,m)=m_3(m,n)$ for $\frac{m}{n}\neq1.$\\

Hence, we have concluded the proof of Theorem \ref{Theorem-Melnikov22}.

\section{Appendix: Pproof of the proposition \ref{prop8}.} \label{Apend}

Let $k_1, k_2, k_3, k_4, k_5$ in $\R$, where
 \begin{enumerate}[(i)]
	 	\item  $k_1$ is the irrational root in $\left(0, \frac{1}{5} \right)$ of  $$q_1(x)=2 x^6-153 x^5+894 x^4-1625 x^3+1234 x^2-444 x+48;$$
	 	\item $k_2,k_3,k_5$ are the irrational roots in $\left(\frac{1}{5},\frac{1}{3} \right),\Big(\frac{3}{2},2 \Big),$ and  $(3,4)$, respectively, of  
  		\[
  		\begin{array}{ll}
  		q_2(x)=&177586560 x^{17}-1447424208 x^{16}+5969663136 x^{15}-29387373904 x^{14}\\&+129626832188 x^{13}-293774330511 x^{12}+102470736381 x^{11}+1027573184492 x^{10}\\&-2645532232771 x^9+3259827826136 x^8-2344796073539 x^7+997977148820 x^6\\&-227233713561 x^5+15757275163 x^4+3311923726 x^3-563207524 x^2\\&+11249208 x+19744;
  		\end{array}
  		\]
	 	
	 	\item  $k_4$  is the irrational root in $(3,4)$ of $$q_3(x)=862 x^5-4831 x^4+8308 x^3-5186 x^2+974 x+1.$$
 \end{enumerate}

By straightforward computations, we obtain
\begin{equation*}
\sgn(\mathcal{A},\mathcal{B},\mathcal{C},\mathcal{D},\mathcal{E})=\left\{
\begin{array}{ll}
(1,1,1,\sgn(\mathcal{D}),-1),& k \in \left(\frac{1}{3},\frac{1}{2}\right),\\
(-1,-1,\sgn(\mathcal{C}),1,1),& k \in \left(\frac{1}{2}, \frac{2}{3}\right),\\
(-1,-1,-1,0,0), & k=\frac{2}{3}, \\
(-1,-1,-1,-1,-1),& k \in \Big( \frac{2}{3},\frac{3}{4} \Big), \\
(1,\sgn(\mathcal{B}),-1,-1,-1),& k \in \Big(\frac{3}{4},k_6 \Big),\\
(1,1,\sgn(\mathcal{C}),-1,-1), & k=k_6, \\
(1,1,1,-1,-1),& k \in  \left(k_6,\frac{4}{5} \right), \\
(1,1,1,0,0), & k=\frac{4}{5}, \\
(1,1,1,1,1),& k \in \left(\frac{4}{5},1 \right) \cup \left(1,\frac{4}{3} \right),\\
(0,0,0,1,1), & k=\frac{4}{3}, \\
(-1,-1,-1,\sgn(\mathcal{D}),1),& k \in \left(\frac{4}{3}, \frac{3}{2} \right), \\
\end{array}\right.
\end{equation*}
where $k_6$  is the zero  in $\left(\frac{3}{4},\frac{4}{5} \right)$ of the polynomial  $240 x^7-610 x^6+885 x^5-1830 x^4+2034 x^3-930 x^2+201 x-38$.
Applying the Descartes rule of signs, we have that $\mathcal{P}$ does not vanish in $(0, \infty)$ for $\Big[\frac{2}{3}, \frac{3}{4} \Big) \cup \left[\frac{4}{5}, 1 \right) \cup \left(1,\frac{4}{3} \right]$ and  $\mathcal{P}$   has one positive zero, which is simple, for $k\in\left(\frac{1}{3},\frac{1}{2} \right) \cup \left(\frac{1}{2}, \frac{2}{3} \right)\cup \left(\frac{3}{4}, \frac{4}{5} \right) \cup \left(\frac{4}{3}, \frac{3}{2} \right).$\\

Now,  by computation the discriminant of $\mathcal{P}$  is given by
\[
	\begin{array}{ll}
	\Dis(\mathcal{P})= -192 (4-3 k)^2 (k-1)^{12} (2 k-1) (3 k-2) (5 k-4) \mathcal{M}\mathcal{N} ,
	\end{array}
	\]
	with
	\[
	\begin{array}{ll}
	\mathcal{M}=&32400 k^{12}-189990 k^{11}+503307 k^{10}-781069 k^9+742059 k^8-352440 k^7-84219 k^6\\&+253647 k^5-161955 k^4+28332 k^3+22032 k^2-15336 k+3296, \\
	\mathcal{N}= &177586560 k^{17}-1447424208 k^{16}+5969663136 k^{15}-29387373904 k^{14}\\&+129626832188 k^{13}-293774330511 k^{12}+102470736381 k^{11}+1027573184492 k^{10}\\&-2645532232771 k^9+3259827826136 k^8 -2344796073539 k^7+997977148820 k^6\\&-227233713561 k^5+15757275163 k^4+3311923726 k^3-563207524 k^2+11249208 k\\&+19744.
	\end{array}
	\]
So,	
\begin{equation}
\sgn(\Dis(\mathcal{P}))=\left\{
\begin{array}{ll}
-1,& k  \in  \left(k_0,\frac{1}{5} \right) \cup \left(\frac{1}{5},k_2 \right) \cup \Big(\frac{3}{2},k_3 \Big)\cup (k_5,\infty), \nonumber \\
1, & k \in  (0,k_0)  \cup \left(k_2,\frac{1}{3} \right) \cup (k_3,k_5).
\end{array}\right.
\end{equation}
Therefore, $\mathcal{P}$ has two real zeros counting multiplicity for $ \left(k_0,\frac{1}{5} \right) \cup \left(\frac{1}{5},k_2 \right) \cup \left(\frac{3}{2},k_3 \right)\cup (k_5,\infty),$ and $P$ has 4 or 0  real zeros counting multiplicit for  $k \in  (0,k_0)  \cup \left(k_2,\frac{1}{3} \right) \cup (k_3,2)\cup (3,k_5).$  In the
sequel, we shall determine the exact number of zeros in  $(0, \infty)$.  Through straightforward computation it is direct to see that
 \begin{equation}
\sgn(\mathcal{A},\mathcal{B},\mathcal{C},\mathcal{D},\mathcal{E})=\left\{
\begin{array}{ll}
(-1, \sgn(B), 1, -1,1),& k \in \left(0, \frac{1}{5} \right), \nonumber \\
(-1, \sgn(B), 1, -1,-1),& k \in \left(\frac{1}{5}, \frac{1}{3}\right), \nonumber \\
(-1, \sgn(B), \sgn(C), -1,-1),& k \in \Big(\frac{3}{2}, 2 \Big), \nonumber \\
(0, 1, 1, -1,-1),& k=2, \nonumber \\
(-1, \sgn(B), 1, -1,-1),& k \in ( 2,3), \nonumber \\
(-1, -1, 1, -1,1),& k \in ( 3,5), \nonumber \\
(0,- 1, 1, -1,1),& k=5, \nonumber \\
(1, -1, 1, -1,1),& k \in ( 5,\infty), \nonumber \\
\end{array}\right.
\end{equation}
and $\mathcal{P}(-1)>0$ for $k \in \left(\frac{1}{5}, k_2 \right) \cup \left(k_2, \frac{1}{3} \right) \cup\Big(\frac{3}{2}, k_3 \Big)\cup (k_3,  3)$, $\mathcal{P}(1)>0$ for $k\in[2,3)$ and $\mathcal{P}(3)<0$ for $k \in (4, \infty)$. Consequently, 
\begin{enumerate}
\item For $k \in \left(k_0,\frac{1}{5} \right) \cup (k_5,5)$, $\mathcal{P}$ has two real zeros, which are simple and one of them is positive and the other is negative;
\item For $k \in \left(\frac{1}{5},k_2 \right) \cup \Big(\frac{3}{2},k_3 \Big)$, $\mathcal{P}$ has two real zeros, which are simple and negative;
\item For $k \in \left(k_2, \frac{1}{3} \right)  \cup (k_3,2)\cup(2,3)$, $\mathcal{P}$ has four real zeros, which are simple and two of them are positive and the other is negative;
\item For $k \in  (0,k_0) \cup (3, k_5) $, $\mathcal{P}$ has four real zeros, which are simple and three of them are positive and the other is negative;
\item For $k=2,5$, $\mathcal{P}$ has one real zero, which is simple and positive;
\item For $k \in (5, \infty)$, $\mathcal{P}$ has two real zeros, which are simple and positive.
\end{enumerate}

This concludes the proof of Proposition \ref{prop8}.



\section*{Declarations}Conflict of interest No conflict of interest exists in the submission of this manuscript, and manuscript is approved by all authors for publication.

\end{document}